\documentclass[a4paper]{amsart}
\usepackage[latin1]{inputenc}
\usepackage{latexsym}
\input amssym.def
\input amssym.tex
\usepackage{color}
\newtheorem{definition}{Definition}[section]
\newtheorem{lemma}[definition]{Lemma}
\newtheorem{theorem}[definition]{Theorem}
\newtheorem{proposition}[definition]{Proposition}

\newtheorem{remark}[definition]{Remark}
\newtheorem{example}[definition]{Example}
\newtheorem{examples}[definition]{Examples}

\def\emph#1{{\bfseries\itshape{#1}}}
 1   
 1  

\def\R{\mathbb{R}}               

\def\lcf{\lbrack\! \lbrack}
\def\rcf{\rbrack\! \rbrack}

\newcount\ancho \newcount\anchom \newcount\anchoa \newcount\anchob \newcount\altura
\newcommand{\ltilde}[3][0]{\altura=0 \advance\altura by #1 \ancho=#2 \anchom=\ancho \divide\anchom by 2
\anchoa=\ancho \divide\anchoa by 4 \anchob=\anchom \advance\anchob by \anchoa \kern-3pt \begin{array}[b]{c} \begin{picture}(1,1)(\anchom,-\altura)
\qbezier(0,2)(\anchoa,5)(\anchom,2) \qbezier(\anchom,2)(\anchob,-1)(\ancho,4) \qbezier(0,2)(\anchoa,4.5)(\anchom,1.8) \qbezier(\anchom,1.8)(\anchob,-1.5)(\ancho,4)
\end{picture} \\[-4pt]{#3} \end{array} \kern-4pt	}

\newcommand{\pl}{(J^T)^{-1} (0,\mu)}

\newcommand{\g}{\mathfrak{g}}

\begin{document}

\title[Reduction of a  symplectic-like Lie algebroid with momentum map]{Reduction of a  symplectic-like Lie algebroid \\
with momentum map and its application to fiberwise linear Poisson structures}

\author{Juan Carlos Marrero}
\address{J.\ C.\ Marrero:
Unidad asociada ULL-CSIC, Geometr{\'\i}a Diferencial y Mec\'anica Geom\'etrica, Departamento de Matem\'atica Fundamental, Facultad de
Matem\'aticas, Universidad de la Laguna, La Laguna, Tenerife,
Canary Islands, Spain} \email{jcmarrer@ull.es}

\author{Edith Padr{\'o}n}
\address{E.\ Padr{\'o}n: Unidad asociada ULL-CSIC, Geometr{\'\i}a Diferencial y Mec\'anica Geom\'etrica,Departamento de Matem\'atica Fundamental, Facultad de
Matem\'aticas, Universidad de la Laguna, La Laguna, Tenerife,
Canary Islands, Spain} \email{mepadron@ull.es}

\author{Miguel Rodr{\'\i}guez-Olmos}
\address{M. Rodr{\'\i}guez-Olmos: Departamento de Matem\'atica Aplicada IV, 
Universidad Polit\'ecnica de Catalu\~na, 
Barcelona, Spain. }
\email{miguel.rodriguez.olmos@upc.edu}

\begin{abstract}

\end{abstract}

\noindent

\thanks{\noindent The authors have been
partially supported by MEC (Spain) grants  MTM2009-13383 and 
MTM2009-08166-E. The research of M.R-O has been also partially supported by a Marie
Curie Intra European Fellowship PIEF-GA-2008-220239 and a Marie Curie
Reintegration Grant PERG-GA-2010-27697. The research of J.C.-M and E-P has been also partially supported by the grants of the Canary government
SOLSUBC200801000238 and ProID20100210. We also would like to thank  D. Iglesias, D. Mart{\'\i}n de Diego and E. Mart{\'i}nez for their useful comments. }

\noindent

\begin{abstract}
This article addresses the problem of developing an extension of
the Marsden-Weinstein reduction process to symplectic-like Lie
algebroids, and in particular to the case of the canonical cover of a fiberwise linear Poisson structure, whose reduction process is the analogue to cotangent bundle reduction in the context of Lie algebroids.
\vspace{1cm}
\begin{center}
{\it Dedicated to the memory of Jerrold E. Marsden}
\end{center}
\end{abstract}

\maketitle

\setcounter{section}{0}

\section{Introduction}

\subsection{Preliminaries}
A smooth and proper action of a Lie group $G$ on a symplectic
manifold $(M,\Omega)$ is called Hamiltonian if $G$ acts by
symplectomorphisms and it admits a coadjoint equivariant momentum
map $J:M\rightarrow\g^*$ satisfying the compatibility condition
$$\Omega(\xi_M,\cdot)=d\langle
J(\cdot),\xi\rangle\quad\forall\xi\in\g,$$ where $\xi_M\in
\mathfrak{X}(M)$ is the fundamental vector field corresponding to
the Lie algebra element $\xi$.  The Marsden-Weinstein symplectic
reduction process, introduced in \cite{MaWe74} states that, if $\mu$
is a regular value of $J$, and $G_\mu$, the stabilizer of $\mu$ for
the coadjoint representation, acts freely and properly on
$J^{-1}(\mu)$, then the quotient $J^{-1}(\mu)/G_\mu$ is a smooth
manifold with a naturally induced ``reduced'' symplectic form
$\Omega_\mu$. The identity characterizing $\Omega_\mu$ is
$$\iota_\mu^*\Omega=\pi_\mu^*\Omega_\mu,$$
where $\iota_\mu:J^{-1}(\mu)\hookrightarrow M$ and
$\pi_\mu:J^{-1}(\mu)\rightarrow J^{-1}(\mu)/G_\mu$ are the natural
inclusion and projection, respectively.

One can look at the particular and important case when our
symplectic structure is not on $M$, but on its cotangent bundle
$T^*M$, and the action of $G$ is the cotangent lift of an action on
$M$. In this case, the lifted action is automatically Hamiltonian
with respect to the canonical symplectic form on $T^*M$ given in
trivializing local coordinates by $\Omega_c=dx^i\wedge dy^i$. It can be
shown that an equivariant momentum map for this action is given by
\begin{equation}\label{cot momentum} J(\alpha_x)(\xi)=
\alpha_x(\xi_M(x))\quad\text{for all}\quad \alpha_x\in
T^*_xM,\,\xi\in\g.\end{equation}

The reduction theory for lifted actions on cotangent bundles was
first studied in \cite{Sat77}, where only the case of Abelian
actions was addressed. The general case was treated in \cite{Ku81}
and \cite{AbMa87}. The set of  results emerging from those and other
references is usually known as cotangent bundle reduction. We will
expose here the basic lines of this subject and refer to \cite{Stages,OR}
for a more detailed survey. We will assume from now on that the
action of $G$ on $M$ is free and proper.

For the case of lifted actions, due to the particularities of the
fibered  geometry existent, we can distinguish different situations for
the choice of momentum value. This cases are 1) $\mu=0$, 2)
$G_\mu=G$ and 3) general values of $J$. The theory of cotangent
bundle reduction establishes the existence of  maps from the
abstract symplectic reduced spaces to certain cotangent bundles
equipped with canonical symplectic forms possibly deformed by  a  {\it  magnetic
term.} The different possibilities are:

\noindent$\bullet$ $\mu=0$. There is a symplectomorphism
$$\phi_0:(J^{-1}(0)/G, \Omega_0)\rightarrow (T^*(M/G),\Omega_c).$$

\noindent$\bullet$ $G_\mu=G$. There is a symplectomorphism
$$\phi_\mu:(J^{-1}(\mu)/G_\mu,\Omega_\mu)\rightarrow
(T^*(M/G_\mu),\Omega_c-B_\mu).$$

\noindent$\bullet$ General $\mu$. There is a symplectic embedding
$$\phi_\mu:(J^{-1}(\mu)/G_\mu,\Omega_\mu)\rightarrow
(T^*(M/G_\mu),\Omega_c-B_\mu).$$

In the last two cases the {\it magnetic term } $B_\mu$ is the pullback
by the cotangent bundle projection of a  closed two-form on
$M/G_\mu$. This two-form is obtained, for example,  via the choice of a principal
connection for the fibration $M\rightarrow M/G_\mu$.

Note also that in the case $G_\mu=G$ (which corresponds to values of
$J$ for which their  coadjoint orbits are trivial) we have
$M/G_\mu=M/G$. Therefore, topologically, all the reduced spaces for
momentum values $\mu$ with trivial coadjoint orbits are equivalent
to the same space $T^*(M/G)$, and their symplectic forms differ only
possibly in the terms $B_\mu$.

This paper develops a generalization of the reduction theory
reviewed above for general symplectic manifolds  and the particular
case of cotangent bundles, to the setup of Lie algebroids. For general
Marsden-Weinstein reduction, this happens when one substitutes the
tangent bundle $TM$ of a symplectic manifold $M$ by a more general
symplectic vector bundle $A$ over $M$ (a symplectic-like Lie algebroid).
For cotangent bundle reduction, the generalization consists in
substituting $TM$ by a general Lie algebroid $A$, and $T(T^*M)$ by a
special construction called the canonical cover (or the prolongation of $A$ over $A^*$ following the terminology of \cite{LMM}) of $A^*$, which happens
to be a symplectic-like Lie algebroid. The latter generalization is a
particular case of the former, and both cases coincide with
Marsden-Weinstein reduction and cotangent bundle reduction,
respectively, when $A$ is just the tangent bundle of $M$. In the
remainder of this section we will give an overview of the new
results of this article.

\subsection{Reduction for  Lie algebroids}
A Lie algebroid is a natural generalization of the tangent bundle to a
manifold. It consists of a vector bundle $A\rightarrow M$ equipped
with a certain geometric structure that allows to generalize on the one
hand, the Lie algebra of vector fields on $M$ to a Lie algebra
structure $\lcf\cdot,\cdot\rcf$ on the space of sections of $A$, and
on the other, the exterior derivative on differential forms to the
a derivation $d^A$ of the exterior algebra of multi-sections of $A^*$. The
general theory of Lie algebroids is reviewed in Section
\ref{section2}. We remark that giving  a Lie algebroid structure on vector bundle $A$ is equivalent to giving  a linear Poisson bivector on  the dual vector bundle $A^*$ of $A$.

In order to study the reduction process for a Lie algebroid $A\to M$ we introduce in Subsection \ref{section3.1} the notion of an action by complete lifts on $A$ as an action $\Phi:G\times A\to A$ of a Lie group $G$  by vector bundle automorphisms of a Lie group $G$ on $A$ together with a Lie algebra anti-morphism $\psi: {\mathfrak g}\to \Gamma(A)$ such that the infinitesimal generator of $\xi\in {\mathfrak g}$ with respect to $\Phi$ is just the complete lift of $\psi(\xi);$ or equivalently,  an action $\Phi:G\times A^*\to A^*$ on the dual vector bundle $A^*$ by Poisson automorphisms such that the infinitesimal  generator of $\xi$ is just the Hamiltonian vector field (with respect to the linear Poisson structure on $A^*$) of the linear function associated with the section $\psi(\xi)\in \Gamma(A).$ The standard example of an action by complete lifts  on the Lie algebroid $TM$ is the tangent lift of an action on $M$.

If $\Phi:G\times A\to A$ is a free and proper action of a connected Lie group $G$ on the Lie algebroid $A$ by complete lifts then in  Section \ref{section reduction liealg} we construct  an affine action $\Phi^T:TG\times A\to A$ of the tangent Lie group $TG$ such that the orbit space $A/TG$ is  a Lie algebroid over the reduced manifold $M/G$ corresponding to the induced action $\phi:G\times M\to M$  of $G$ on the base manifold $M$ of the Lie algebroid $A$. Moreover, we prove that the projection $\widetilde{\pi}:A\to A/TG$ is a Lie algebroid morphism (see Theorem \ref{reduction}).

\subsection{Reduction for symplectic-like Lie algebroids}
The main idea behind the generalization of
symplectic reduction to Lie algebroids consists in realizing that a
symplectic manifold can be seen as a Lie algebroid endowed with a
symplectic vector space structure on each fiber varying smoothly.
Under this point of view, the Lie algebroid is nothing but the
tangent bundle of the symplectic manifold, and the symplectic
structure on the fibers is the evaluation  of the symplectic form to
each point. The fact that the symplectic form is closed can then be
interpreted as being closed as a differential two-form on the Lie
algebroid. This situation can be extended to an arbitrary Lie
algebroid, not necessarily the tangent bundle of a symplectic manifold.
Therefore, the setup for this paper will be a symplectic-like Lie
algebroid, i.e. a Lie algebroid $A\rightarrow M$ equipped with a
non-degenerate smooth  2-section $\Omega\in\Gamma(\wedge^2A^*)$
satisfying $d^A\Omega=0$ and an action $\Phi:G\times A\to A$ of a Lie group $G$  by complete lifts on $A$. 

The main result of Subsection \ref{section3.2} is to obtain a Lie algebroid version of the Marsden-Weinstein reduction for symplectic manifolds. Firstly, we will consider a momentum map $J:M\to {\mathfrak g}$ for the action $\phi:G\times M\to M$ which allows to define an equivariant map $J^T:A\to {\mathfrak g}^*\times {\mathfrak g}^*$ for  the affine action $\Phi^T:TG\times A\to A.$ Then, in Theorem \ref{2.8co}, we describe the Lie
algebroid analogue of the Marsden-Weinstein reduction scheme. It
states that under a regularity condition involving a value $\mu\in {\mathfrak g}^*$ of
$J$, the quotient $A_\mu:=({J^T})^{-1}(0,\mu)/TG_\mu$ is a symplectic-like
Lie algebroid over $J^{-1}(\mu)/G_\mu$.  If $\Omega$ is the
symplectic-like section on $A$ and
$\widetilde{\pi}_\mu:({J^T})^{-1}(0,\mu)\rightarrow A_\mu$ and
$\widetilde{\iota}_\mu:({J^T})^{-1}(0,\mu)\rightarrow A$ are the
canonical projection and inclusion respectively, then the reduced
symplectic-like section $\Omega_\mu$ on $A_\mu$ is characterized by the
condition
$$\widetilde{\pi}_\mu^*\Omega_\mu=\widetilde{\iota}_\mu^*\Omega.$$

It is well-known that the base manifold of a symplectic-like Lie algebroid has an induced
Poisson structure (see \cite{LMM, K, MX}). Then, as a consequence of the reduction theorem for
symplectic-like Lie algebroids, it is shown in Theorem \ref{T2.10} that
the Poisson structures on the base manifolds of the original and
reduced symplectic-like Lie algebroids are related in a similar way. Namely,
if $\{\cdot,\cdot\}$ denotes the Poisson structure on $M$ induced by
$\Omega$ and $\{\cdot,\cdot\}_\mu$ is the corresponding structure on $J^{-1}(\mu)/G_\mu$ 
induced by the reduced symplectic-like section $\Omega_\mu$, then

$$\{ \tilde{f},\tilde{g}\}_\mu\circ \pi_\mu=\{f,g\}\circ i_\mu,$$

where $\pi_\mu:J^{-1}(\mu)\rightarrow J^{-1}(\mu)/G_\mu$ and
$\iota_\mu:J^{-1}(\mu)\rightarrow M$ are the canonical projection
and the  inclusion, respectively, $\tilde{f}, \tilde{g}$ are functions on $J^{-1}(\mu)/G_\mu$  and $f,g$ are $G$-invariant  extensions to $M$ of ${\tilde
f}\circ \pi_\mu$ and ${\tilde g}\circ \pi_\mu$, respectively. That is, the reduction obtained on the base manifold of the Lie algebroid is just the Marsden-Ratiu reduction for Poisson manifolds \cite{MR}.

In \cite{BCG} a theory of reduction for Courant algebroids is presented. A symplectic-like Lie algebroid $A$ induces a Lie bialgebroid and therefore a Courant algebroid on $A\oplus A^*$ (see \cite{LWP}). Then,  one may apply this Courant reduction process  to  $A\oplus A^*$ and could recover, after a long computation,  some results described in Section \ref{section3.2}. However, we focus our study in the reduction of the particular case of symplectic-like Lie algebroids which allows us to obtain more explicit results on this type of reduction.

\subsection{Reduction for canonical covers of fiberwise linear Poisson structures}

Section \ref{section prolongued} studies, within the framework of
symplectic-like Lie algebroids, the situation equivalent to cotangent
bundle reduction. In this case the generalization goes as follows:
First, the cotangent bundle over a manifold $M$ is replaced by
$A^*$, the dual of a Lie algebroid $A\rightarrow M$, and then we
consider the canonical cover of $A^*$, (also known as the prolongation of $A$ over $A^*$), denoted by  ${\mathcal T}^AA^*$.
This is a natural construction on the dual of a Lie algebroid, which
happens to be in a canonical way, a symplectic-like Lie algebroid with base manifold
$A^*$. If $A$ is the tangent bundle of $M$, then ${\mathcal T}^AA^*$
is just $T(T^*M)$. If there is a suitable action of a Lie group $G$ by complete lifts on $A$, this action can be further lifted to the
canonical cover of $A^*,$ in a natural way, and this lifted action
happens to be a morphism of symplectic-like Lie algebroids. Furthemore, one may define an equivariant momentum map on $A^*$ (the base space of ${\mathcal T}^AA^*$) in a similar way when as how the classical momentum map \eqref{cot momentum} on $T^*M$ is introduced.   In general it is not possible to find an equivariant momentum map for a Poisson action (see, for instance, \cite{FRO}). However, for the case of the Poisson action $\Phi^*: G\times A^*\to A^*$ associated with an action $\Phi:G\times A\to A$ by complete lifts, an equivariant momentum map is described.

Applying the reduction
theory of symplectic-like Lie algebroids just developed we know that the reduction of
${\mathcal T}^AA^*$ at any momentum value is again a symplectic-like Lie
algebroid. However, as in the situation of cotangent bundle
reduction it is expected that the extra properties of the symplectic-like
Lie algebroid, in this case the prolonged fibered structure, will be
recovered in the quotient in some way. This is the content of the
results of Section 5, for which the obtained new results reduce to  the
standard cotangent bundle reduction theory in the case that the
starting Lie algebroid $A\rightarrow M$ is the standard Lie algebroid $TM$. In Subsection
\ref{prol1} it is shown (Theorem \ref{p1}) that if $\mu=0$, there is a symplectic-like Lie
algebroid isomorphism between the reduced symplectic-like Lie algebroid
$({J^T})^{-1}(0,0)/TG$ and ${\mathcal T}^{A_0}A_0^*$. Here
$A_0\rightarrow M/G$ is a Lie algebroid with total space $A/TG.$
 The case $G_\mu=G$ is studied in Subsection \ref{prol2}. There
it is shown that $({J^T})^{-1}(0,\mu)/TG_\mu$ is also isomorphic to
${\mathcal T}^{A_0}A_0^*$, but in this case this isomorphism is canonical between the symplectic-like Lie algebroids  if the canonical symplectic-like section on ${\mathcal
T}^{A_0}A_0^*$ is modified by the addition of a twisting term which
consists in the lift to ${\mathcal T}^{A_0}A_0^*$ of a closed
2-section of $A_0^*$. This is the content of Theorem \ref{p2}.
Finally, Subsection \ref{prol3}, in its main result, Theorem \ref{t5.3} shows that for the most
general momentum values,  $({J^T})^{-1}(0,\mu)/TG_\mu$ is
canonically embedded as a Lie subalgebroid of ${\mathcal
T}^{A_{0,\mu}}A_{0,\mu}^*$, where $A_{0,\mu}$ is a Lie algebroid $A/TG_\mu$ over $M/G_\mu$, and the prologation ${\mathcal
T}^{A_{0,\mu}}A_{0,\mu}^*$ is equipped with its canonical symplectic-like section
minus a magnetic term, just as in the $G_\mu=G$ case.  

As far as we know there is a similar research being done independently by E. Mart{\'\i}nez \cite{Mar2}. In addition, in the same direction, some similar results in the more general setting of Lie bialgebroids has been discussed in \cite{Ro}.

\section{Lie algebroids}
\label{section2}
 Let $A$ be a vector bundle of rank $n$
over the manifold $M$ of dimension $m$ and let $\tau:A\rightarrow M$
be its vector bundle projection. Denote by $\Gamma(A)$ the
$C^{\infty}(M)$-module of sections of $\tau:A\rightarrow M$. A
{\it Lie algebroid structure} $(\lcf\cdot,\cdot\rcf,\rho)$ on $A$
is a Lie bracket $\lcf\cdot,\cdot\rcf$ on the space $\Gamma(A)$
and a bundle map $\rho:A\rightarrow TM$, called {\it the anchor
map}, such that, if we also denote by
$\rho:\Gamma(A)\rightarrow\frak{X}(M)$ the homomorphism of
$C^{\infty}(M)$-modules induced by the anchor map satisfying 
\begin{equation}\label{IL}
\lcf X,fY\rcf=f\lcf X,Y\rcf+\rho(X)(f)Y, \mbox{ for }
X,Y\in\Gamma(A) \mbox{ and } f\in C^{\infty}(M). \end{equation}
The triple $(A,\lcf\cdot,\cdot\rcf,\rho)$ is called a {\it Lie
algebroid over $M$} (see \cite{Ma}). In such a case, the anchor
map $\rho:\Gamma(A)\rightarrow\frak{X}(M)$ is a homomorphism
between the Lie algebras $(\Gamma(A),\lcf\cdot,\cdot\rcf)$ and
$(\frak{X}(M),[\cdot,\cdot])$.

If $(A,\lcf\cdot,\cdot\rcf,\rho)$ is a Lie algebroid, one can
define a cohomology operator, which is called {\it the
differential of $A$},
$d^A:\Gamma(\wedge^kA^*)\longrightarrow\Gamma(\wedge^{k+1}A^*)$,
as follows
\begin{equation}\label{diferencial}
\begin{array}{lcl}
(d^A\mu)(X_0,\dots,X_k)&=&\displaystyle\sum_{i=0}^k(-1)^i\rho(X_i)(\mu(X_0,\dots,\widehat{X_i},\dots,X_k))\\
 &+&\displaystyle\sum_{i<j}(-1)^{i+j}\mu(\lcf
 X_i,X_j\rcf,X_0,\dots,\widehat{X_i},\dots,\widehat{X_j},\dots,X_k),
\end{array}
\end{equation}
for $\mu\in\Gamma(\wedge^k A^*)$ and $X_0,\dots,X_k\in\Gamma(A)$.
Moreover, if $X\in\Gamma(A)$ one may introduce, in a natural way,
{\it the Lie derivate  for multisections of $A^*$ with respect to $X$,} as the operator
${\mathcal L}_X^A:\Gamma(\wedge^k
A^*)\longrightarrow\Gamma(\wedge^{k} A^*)$ given by ${\mathcal
L}_X^A=i_X\circ d^A+d^A\circ i_X$ (see \cite{Ma}). {\it The Lie derivative of a multisection $P\in 
\Gamma(\wedge^k A)$ of $A$ with respect to $X$ } is the $k$-section ${\mathcal L}^A_XP$ on $A$ characterized by  
$${\mathcal L}^A_XP(\alpha_1, \dots,\alpha_k)=\rho(X)(P(\alpha_1,\dots ,\alpha_k))-\sum_{i}P(\alpha_1,\dots, {\mathcal L}^A_X\alpha_i,\dots ,\alpha_k)$$
with $\alpha_i\in \Gamma(A^*).$

If $(A,\lcf \cdot,\cdot \rcf, \rho )$ is a Lie algebroid, we have a natural  linear Poisson structure $\Pi _{A^*}$ on the dual
vector bundle $A^*$ characterized as follows:
\begin{equation}\label{P-lineal}
\begin{array}{l}
\{ \widehat{X},\widehat{Y}\}_{\Pi _{A^*}} =-\widehat{\lcf X ,Y
\rcf},\\[9pt] \{\widehat{X},f_M\circ \tau _\ast \}_{\Pi_{A^*}}=
- \rho (X)(f_M)\circ \tau _\ast   ,\\[9pt]
\{f_M\circ \tau _\ast ,h_M\circ \tau _\ast \}_{\Pi_{A^*}}=0,
\end{array}
\end{equation}
for $X,Y \in \Gamma (A)$ and $f_M,h_M\in C^\infty (M)$, $\tau
_\ast :A^\ast \to M$ being the canonical projection. Here,
$\widehat{X}$ and  $\widehat{Y}$ denote the linear functions on
$A^*$ induced by the sections $X$ and $Y$, respectively.
Conversely, if $A^*$ is endowed with a linear Poisson structure
$\Pi _{A^*}$, then it induces a Lie algebroid structure on $A$
characterized by \eqref{P-lineal} (see
\cite{Courant}).

\medskip
Now, suppose that $(A,\lcf\cdot,\cdot\rcf,\rho)$ and
$(A',\lcf\cdot,\cdot\rcf',\rho')$ are Lie algebroids over $M$ and
$M'$, respectively, and that $F:A\to A'$ is a vector bundle
morphism over the map $f:M\to M'.$ Then, $F$ is said to be a {\it
Lie algebroid morphism} if
\begin{equation}\label{dAF}
d^{A}(F^*\alpha')=F^*(d^{A'}\alpha'),\mbox{ for
}\alpha'\in\Gamma(\wedge^k(A')^*) \mbox{ and for all } k.
\end{equation}
Here $F^*\alpha'$ denotes  the section of the vector bundle
$\wedge^k A^*\rightarrow M$ defined by
\begin{equation}
(F^*\alpha')_{x}(a_1,\dots,a_k)=\alpha'_{f(x)}(F(a_1),\dots,F(a_k)),
\end{equation}
for $x\in M$ and $a_1,\dots,a_k\in A_{x}$. 

If $F:A\to A'$ is  a vector bundle isomorphism over a diffeomorphism $f:M\to M'$ then  the dual isomorphism $F^*:(A')^*\to A^*$  over $f^{-1}: M'\to M$ is defined as follows 
$$[F^*(\alpha'_{x'})](a_{f^{-1}(x')})=\alpha'_{x'}(F(a_{f^{-1}(x')})),$$
for $x'\in M',$ $\alpha'_{x'}\in (A')^*_{x'}$ and $a_{f^{-1}(x')}\in A_{f^{-1}(x')}.$ 

Moreover, we have that $F$ is a Lie algebroid isomorphism if and only if $F^*$ is a Poisson isomorphism, that is, 
\[
\{f'\circ F^*, g'\circ F^*\}_{\Pi_{A^*}}=\{f',g'\}_{\Pi_{(A')^*}}\circ F^*, \mbox{ for } f',g'\in C^\infty((A')^*).
\]
If $F$ is 
a Lie
algebroid morphism, $f$ is an injective immersion and
$F_{|A_x}:A_x\rightarrow A'_{f(x)}$ is injective, for all $x\in
M$, then $(A,\lcf\cdot,\cdot\rcf,\rho)$ is said to be a {\it Lie
subalgebroid} of $(A',\lcf\cdot,\cdot\rcf',\rho')$.

Let $\widetilde{\pi}:A\to A'$ be an epimorphism of vector bundles
over $\pi:M\to M'$, i.e. $\pi$ is a submersion and for
each $x\in M,$ $\widetilde{\pi}_x:A_x\to A'_{\pi(x)}$ is an
epimorphism of vector spaces. If $X:M\to A$ is a section of $A$,
we said that $X$ is {\it $\widetilde{\pi}$-projectable } if there
is $X'\in \Gamma(A')$ such that the following diagram is
commutative

\begin{picture}(375,60)(60,40)
\put(210,20){\makebox(0,0){$M$}} \put(260,25){$\pi$}
\put(215,20){\vector(1,0){90}} \put(315,20){\makebox(0,0){$M'$}}
\put(190,50){$X$} \put(210,30){\vector(0,1){40}}
\put(320,50){$X'$} \put(310,30){\vector(0,1){40}}
\put(210,80){\makebox(0,0){$A$}} \put(260,85){$\widetilde{\pi}$}
\put(215,80){\vector(1,0){90}} \put(315,80){\makebox(0,0){$A'$}}
\end{picture}

\vspace{30pt}

In the next proposition we will describe the
necessary and sufficient conditions to obtain a Lie algebroid
structure on $A'$ such that $\widetilde{\pi}$ is a Lie algebroid
morphism.

\begin{proposition}\label{p-reduction} 
\mbox{\cite{IMMMP}} Let $(A,\lcf\cdot,\cdot\rcf,\rho)$ be a Lie algebroid and  $\widetilde{\pi}:A\to A'$  an epimorphism of vector bundles. Then, there is a Lie algebroid structure on $A'$
such  that $\widetilde{\pi}$ is a Lie algebroid epimorphism if and
only if the following conditions hold:
\begin{enumerate}
\item $\lcf X,Y\rcf$ is a $\widetilde{\pi}$-projectable section of
$A$, for all $X,Y\in \Gamma(A)$ $\widetilde{\pi}$-projectable
sections of $A$.
\item $\lcf X,Y\rcf\in \Gamma(\ker \widetilde{\pi})$, for all
$X,Y\in \Gamma(A)$ with $X\in \Gamma(A)$ a
$\widetilde{\pi}$-projectable section of $A$ and $Y\in
\Gamma(\ker\widetilde{\pi}).$
\end{enumerate}
\end{proposition}
An equivalent dual version of this result was proved in \cite{CNS}.

Let $X$ be a section of the Lie algebroid $A$. {\it The vertical lift } of $X$ is
the vector field on $A$ given by $X^v(a)=X(\tau(a))_a^v$ for $a\in
A,$ where $\;^v_a:A_{\tau(a)}\to T_a(A_{\tau(a)})$ is the
canonical isomorphism of vector spaces.

On the other hand, there is a unique vector field $X^c$ on $A$,
{\it the complete lift of $X$ to $A$}, such that $X^c$ is
$\tau$-projectable on $\rho(X)$ and
$X^c(\widehat{\alpha})=\widehat{{\mathcal L}_X^A\alpha},$ for all
$\alpha\in \Gamma(A^*)$ (see \cite{GU1,GU2}). Here
$\widehat{\beta}$, with $\beta\in \Gamma(A^*),$ is the linear
function on $A$ induced by $\beta.$

We have that, for all $X,Y\in \Gamma(A),$
\begin{equation}\label{cv} \begin{array}{c}
[X^c,Y^c]=\lcf X,Y\rcf^c,\;\;\; [X^c,Y^v]=\lcf X,Y\rcf^v, \;\;\;
[X^v,Y^v]=0.\end{array}
\end{equation}

The flow of $X^c\in {\frak X}(A)$ is related with Lie algebroid
structure of $A$ as follows.

\begin{proposition}\label{Poisson}\mbox{\cite{R,Mar}}
Let $(A,\lcf\cdot,\cdot\rcf,\rho)$ be a Lie algebroid over $M$ and 
$X$ a section of $A$. Then, for all $P\in \Gamma(\wedge^k
A)$ (respectively, $\alpha\in \Gamma(\wedge^k A^*)$)
\begin{enumerate}
\item[$(i)$] There exists a local flow $\varphi_s:A\to A$ which covers  smooth maps $\bar{\varphi}_s:M\to M$ such
that
\begin{equation}\label{flow}
{\mathcal L}_X^AP=\frac{d}{ds}((\varphi_s)_*P)_{|s=0},\;\;\;
(\mbox{respectively, } {\mathcal
L}_X\alpha=\frac{d}{ds}((\varphi_s)^*\alpha)_{|s=0})
\end{equation}
\item[$(ii)$] ${\mathcal L}_X^AP=0$ if and only if $(\varphi_s)_*P=P.$
\item[$(iii)$] ${\mathcal L}_X^A\alpha=0$ if and only if
$\varphi_ s^*\alpha=\alpha$.
\item [$(iv)$] The vector field $X^c$ on $A$ is complete if and
only if the vector field $\rho(X)$ on $M$ is complete.
\end{enumerate}
\end{proposition}
Here $(\varphi_s)_*P$ is the section of $\wedge^kA\to M$ defined by 
$$((\varphi_s)_*P)(x)(\alpha_1,\dots, \alpha_k)=P(\bar{\varphi}_s^{-1}(x))(\varphi_s^*(\alpha_1),\dots \varphi_s^*(\alpha_k))$$
for all $x\in M$ and $\alpha_1,\dots , \alpha_k\in A^*_x.$

If $X$ is a section of $A$ we define {\it the complete
lift } of $X$ to $A^*$, as the vector field $X^{*c}$ on $A^*$
which is $\tau^*$-projectable on $\rho(X)$ and
$X^{*c}(\widehat{Y})=\widehat{\lcf X,Y\rcf},$ for all
$Y\in\Gamma(A)$  (see \cite{LMM}). If $\{\varphi_s\}$ is the local flow of $X^c$ then the local flow of $X^{*c}$ is $\{\varphi_{-s}^*\}.$ 

If $(x^i)$ are local coordinates on $M$ and $\{e_I\}$ is a local
basis of sections of $A$, then we have the local functions
$\rho_I^i$, $C_{IJ}^K$, ({\it the structure functions of A}) on
$M$ which are characterized by
\[
\rho(e_I)=\rho_I^i\frac{\partial }{\partial x^i},\;\;\;\; \lcf
e_I,e_J\rcf=C_{IJ}^K e_K.
\]
If $(x^i, y^I)$ (respectively, $(x^i, y_I)$) denote the local
coordinates on $A$ (respectively, $A^*$) induced by the local
basis $\{e^I\}$ (respectively, the dual basis $\{e_I\}$) then, for
a section $X=X^I e_I$ of $A,$ the vector fields  $X^v$, $X^c$ and
$X^{*c}$
 are given by
\begin{eqnarray}\label{levantamiento}
X^v=X^I\frac{\partial }{\partial y^I},\;\;\;
X^c=X^I\rho_I^i\frac{\partial }{\partial x^i} +
\big(\rho_J^i\frac{\partial X^I}{\partial x^i}- X^K C_{KJ}^I
\big)y^J\frac{\partial }{\partial y^I},\\
X^{*c}=X^I \rho_I^i\frac{\partial}{\partial x^i}-(\rho_I^i\frac{\partial X^K}{\partial x^i}+
C_{IJ}^KX^J)y_K\frac{\partial }{\partial y_I}.
\end{eqnarray}

\setcounter{equation}{0}
\section{Reduction of symplectic-like Lie algebroids in the presence of a
momentum map}\label{section reduction liealg}
\subsection{Actions by complete lifts  for
Lie algebroids} \label{section3.1}

Let $(A,\lcf\cdot,\cdot\rcf,\rho)$ be a Lie algebroid over the
manifold $M$ and let $\tau:A\to M$ be the corresponding vector bundle
projection. We consider an  left action $\Phi:G\times A\to A$ by vector
bundle automorphisms of a connected Lie group $G$ on $A$. Then, $\Phi$ induces a linear left action $\Phi^*:G\times A^*\to A^*$  given by
$$(\Phi^*_g)_{|A_x^*}=((\Phi_{g^{-1}})_{|A_{\phi_g(x)}})^*:A_x^*\to A_{\phi_g(x)}^*,\,\;\; \mbox{ for } g\in G \mbox{ and } x\in M,$$
where $\phi:G\times M\to M$ is the corresponding action of $G$ on $M.$

We say that $\Phi:G\times A\to A $ is an action {\it by complete
lifts} if there is a Lie algebra anti-morphism $\psi:{\frak g}\to
\Gamma(A)$ such that the infinitesimal generator of $\xi\in {\frak
g}$ with respect to $\Phi$ is just the complete lift of
$\psi(\xi)$ to $A$. Note that this condition implies that
$\xi_M=\rho(\psi(\xi))$, where $\xi_M$ is the infinitesimal
generator of the action $\phi:G\times M\to M$  with respect to $\xi$. Moreover, $\psi(\xi)^c$
is a morphic vector field in the sense of \cite{MX2} and
therefore, for all $g\in G$, $\Phi_g:A\to A$ is a Lie algebroid
automorphism. Thus, the induced action $\Phi^*:G\times A^*\to A^*$ of $G$ on $A^*$ is Poisson with respect to the corresponding linear Poisson structure on $A^*.$ Furthermore, we have the following result.  

\begin{proposition}\label{AP}
Let $(A,\lcf\cdot,\cdot\rcf,\rho)$ be a Lie algebroid on the manifold $M,$ $\Phi:G\times A\to A$  an action by vector bundle automorphisms of a connected Lie group $G$ on $A$ and $\psi:{\mathfrak g}\to \Gamma(A)$  a Lie algebra anti-morphism. Then, 
 $\Phi:G\times A\to A$ is an action of the Lie group $G$ on $A$  by  complete lifts with respect to $\psi$ if and only if   
 $\Phi^*:G\times A^*\to A^*$ is an action on $A^*$ by Poisson morphisms such that the infinitesimal generator $\xi_{A^*}$ associated with $\xi\in {\mathfrak g}$ is just the Hamiltonian vector field corresponding to the linear function $\widehat{\psi(\xi)}$ associated with the section 
 $\psi(\xi)\in \Gamma(A).$ 
\end{proposition}
\begin{proof}
Denote by $\Pi_{A^*}$ the linear Poisson structure on $A^*.$ We will prove that the Hamiltonian vector field $$H_{\widehat{\psi(\xi)}}^{\Pi_{A^*}}=i_{d\widehat{\psi(\xi)}}\Pi_{A^*}\in {\mathfrak X}(A^*)$$ is just the infinitesimal generator $\xi_{A^*}\in {\mathfrak X}(A^*)$  of $\xi$ with respect to the action $\Phi^*$. In fact, if $f\in C^\infty(M)$ and $X\in \Gamma(A),$ using (\ref{P-lineal}), we have that
$$
\begin{array}{rcl}
H_{\widehat{\psi(\xi)}}^{\Pi_{A^*}}(f\circ \tau_{*})&=&\{f\circ \tau_{*},\widehat{\psi(\xi)}\}_{\Pi_{A^*}}=\rho(\psi(\xi))(f)\circ \tau_{*}=(\psi(\xi))^{*c}(f\circ \tau_{*})\\
H_{\widehat{\psi(\xi)}}^{\Pi_{A^*}}(\widehat{X})&=&\{\widehat{X},\widehat{\psi(\xi)}\}_{\Pi_{A^*}}=\widehat{\lcf \psi(\xi),X\rcf}=
(\psi(\xi))^{*c}(\widehat{X}).
\end{array}
$$
Here, $\{\cdot,\cdot\}_{\Pi_{A^*}}$ is the Poisson bracket associated with ${A^*}.$ Thus, $H_{\widehat{\psi(\xi)}}^{\Pi_{A^*}}=(\psi(\xi))^{*c}$. 

On the other hand,   the flow of $(\psi(\xi))^{c}\in {\frak X}(A)$ is $\{\Phi_{exp(t\xi)}:A\to A\}_{t\in \R}$ if and only if  the flow of $(\psi(\xi))^{*c}\in {\mathfrak X}( A^*)$ is $\{\Phi^*_{exp(-t\xi)}:A^*\to A^*\}_{t\in \R}$ and,  in consequence,  we have that the proposition holds. 

\end{proof}


\begin{examples}\label{examples}{\rm
$(i)$  If $A=TM$ and $\Phi=T\phi$ is the tangent lift
of the action $\phi:G\times M\to M$ then it is clear that $\Phi$
is an action by complete lifts with Lie anti-morphism 
$$\psi: {\mathfrak g}\to {\mathfrak X}(M),\;\;\; \psi(\xi)=\xi_M.$$

\medskip

$(ii)$ Let $G$  be Lie group.   If $({\mathfrak g},[\cdot ,\cdot]_{\mathfrak g})$ is the Lie algebra associated with $G$, then we have, in a natural way,  a Lie algebroid structure on ${\mathfrak g}\times TM\to M,$ where the Lie bracket and the anchor map are characterized by 
$$ [(\xi_1,X_1), (\xi_2,X_2)]=([\xi_1,\xi_2]_{\mathfrak g},[X_1,X_2]),\;\;\;\; \rho(\xi,X)=X$$
for all $\xi_1,\xi_2,\xi\in {\mathfrak g}$ and $X_1,X_2,X\in {\mathfrak X}(M).$ 

Now, consider  a free and proper action $\phi:G\times M\to M$ of $G$ on the manifold $M$. We denote by  $\Phi:G\times ({\mathfrak g}\times TM)\to  {\mathfrak g}\times TM$ and $\psi:{\mathfrak g}\to \Gamma({\mathfrak g}\times TM)\cong C^\infty(M,{\mathfrak g})\times {\mathfrak X}(M)$ 
the action of $G$ on ${\mathfrak g}\times TM$ and the Lie algebra anti-morphism, respectively,  given by 
$$\Phi_g(\xi,v_x)=(Ad^G_{g}\xi,T_x\phi_g(v_x)),\;\;\; \xi\in {\mathfrak g} \mbox{ and } v_x\in T_xM$$
 $$\psi(\bar{\xi})=(-\bar{\xi}, \bar{\xi}_M),\;\; \bar{\xi}\in {\mathfrak g},$$ 
 where $Ad^G:G\times {\mathfrak g}\to {\mathfrak g}$ is the adjoint action of $G$ on ${\mathfrak g}.$ 
Note that  if $ad^G_{{\bar{\xi}}}:{\mathfrak g}\to {\mathfrak g}$ denotes  the infinitesimal generator of the adjoint action for  ${{\bar{\xi}}}\in {\mathfrak g},$  then the infinitesimal generator of $\bar{\xi}\in {\mathfrak g}$ with respect $\Phi$ is just $(ad^G_{\bar{\xi}},\bar{\xi}_M^c).$ 
 Thus, $\Phi$ is a free and proper action by complete lifts with respect to $\psi.$ 
  }
 \end{examples}
\begin{remark}\label{phi}
{\rm Suppose that we have an action $\Phi:G\times A\to A$ of a Lie group $G$ on a Lie algebroid $A$ by complete lifts with respect to $\psi:{\mathfrak g}\to \Gamma(A)$ such that the corresponding action $\phi:G\times M\to M$ on $M$ is free.  In such a case, for all $x\in M$, $\psi_x:{\mathfrak g}\to A_x$ is injective. Indeed, if $\xi,\xi'\in {\mathfrak g}$ satisfy $\psi_x(\xi)=\psi_x(\xi')$, we have that $\xi_{M}(x)=\rho(\psi_x(\xi))=\rho(\psi_x(\xi'))=\xi'_{M}(x)$ which implies, using the fact that $\phi$ is free, that $\xi=\xi'.$}
\end{remark}

Next, we will prove that each action of a connected Lie group $G$
over a Lie algebroid $A$ by complete lifts induces an affine
action of the Lie group $TG$ over $A$. Previously, we recall some
facts which are related with the Lie group structure of $TG.$

If $G$ is a Lie group then $TG$ is also a Lie group. In fact, if
$\cdot:G\times G\to G$ denotes the multiplication of $G$, then the
tangent map $T\cdot:TG\times TG\to TG$ of $\cdot$ is such that
$(TG,T\cdot)$ is a Lie group. Moreover, $TG$ may be identified
with the cartesian product $G\times {\frak g}$, where $({\frak
g},[\cdot,\cdot]_{\frak g})$ is the corresponding  Lie algebra
associated with $G$. This identification is given by
\[
TG\to G\times {\frak g},\;\;\;\; X_g\in T_gG\to
(g,(T_gl_{g^{-1}})(X_g))\in G\times {\frak g},
\]
$l_{g^{-1}}:G\to G$ being the left translation by $g^{-1}$ on $G$.
The corresponding Lie group structure on $G\times {\frak g}$ is
defined as follows
\begin{equation}\label{A0}
(g,\xi)\cdot (g',\xi')=((g\cdot g'),\xi' + Ad^G_{(g')^{-1}}\xi)
\end{equation}
and its associated Lie algebra is ${\frak g}\times {\frak g}$ with
the Lie bracket
\begin{equation}\label{prodal}
[(\xi,\eta),(\xi',\eta')]_{{\frak g}\times {\frak
g}}=([\xi,\xi']_{\frak g}, [\xi,\eta']_{\frak
g}-[\xi',\eta]_{\frak g}).
\end{equation}
Here $Ad^G:G\times {\frak g}\to {\frak g}$ denotes the adjoint
action of $G$.

Moreover, if $Coad^{TG}:(G\times \frak g)\times (\frak g^*\times
\frak g^*) \to \frak g^* \times \frak g^*$ is the left coadjoint  action
of $TG\cong G\times \frak g$ on the dual space of the Lie algebra
$(\frak g\times \frak g, [ \cdot ,\cdot ]_{\frak g\times \frak
g})$, then
\begin{equation}\label{coadjunto-TG}
Coad ^{TG}_{(g,\xi )} (\mu ',\mu '')=(Coad^G_g(\mu '+coad^G_\xi
\mu ''),Coad^G _g\mu ''),
\end{equation}
for $(g,\xi )\in G\times \frak g$ and $(\mu ',\mu '')\in \frak
g^*\times \frak g^*$, where $Coad ^G:G\times \frak g^*\to\frak
g^*$ is the left coadjoint action associated with $G$ and $coad^G:\frak
g\times \frak g^*\to \frak g^*$ is the corresponding infinitesimal left 
coadjoint  action.

The following proposition describes how $\psi$ works  with
respect to the action $\Phi$.

\begin{proposition}\label{prop3.2}
Let $\Phi:G\times A\to A$ be an action of a connected Lie group
$G$ on the Lie algebroid $A$ by complete lifts with respect to
$\psi:{\frak g}\to \Gamma(A)$. Then,
\begin{equation}\label{eqlemma}
\Phi_g(\psi(Ad^G_{g^{-1}}\xi)(x))=\psi(\xi)(\phi_g(x))
\end{equation}
for all $\xi\in {\frak g},$ $g \in G$ and $x\in M$.
\end{proposition}

\begin{proof}
 We organize the proof in two steps.

{\it First step:} Suppose that $G$ is a connected and simply
connected Lie group. Consider the map
\[
\psi ^{cv}:{\frak g}\times {\frak g}\to {\frak X}(A),\;\;\; \psi ^{cv}(\xi,\eta)=(\psi(\xi))^c + (\psi(\eta))^v.
\]
Using (\ref{cv}), we may prove easily that $\psi^{cv}$ is an
infinitesimal action of $TG$ over $A$, that is, $\psi^{cv}$ is
$\R$-linear and
\[
\psi ^{cv}([(\xi,\eta),(\xi',\eta')]_{{\frak g}\times {\frak
g}})=[\psi ^{cv}(\xi,\eta), \psi ^{cv}(\xi',\eta')].
\]
Then, since the vector field $\psi ^{cv}(\xi,\eta)\in {\frak X}(A)$
is complete, from Palais Theorem (see \cite{Palais}), there is a
unique action $\Phi ^T:TG\times A\to A$ from $TG\cong G\times
{\frak g}$ such that for all $(\xi,\eta)\in {\frak g}\times {\frak
g}$
\[
(\xi,\eta)_A=\psi ^{cv}(\xi,\eta)=(\psi(\xi))^c + (\psi(\eta))^v.
\]

Here $(\xi,\eta)_A\in {\frak X}(A)$ is the infinitesimal generator
of $(\xi,\eta)$ with respect to the action $\Phi^T.$

Now, suppose that $g=exp_G(\eta)$. Then, we have that
\begin{equation}\label{l1}
\Phi_g(\psi(Ad^G_{g^{-1}}\xi)(x))=\Phi ^T((g,0_{\frak
g}),\psi(Ad^G_{g^{-1}}\xi)(x))
\end{equation}
with $0_{\frak g}$ being the zero of ${\frak g}.$ In fact, one can
prove that
\begin{equation}\label{l3}
\pi:\R\to G\times {\frak g},\;\;\; s\mapsto \pi(s)=
(exp_G(s\eta),0_{\frak g}) \end{equation} is an one-parameter
subgroup and $\frac{d\pi}{ds}_{|s=0}=(\eta,0_{\frak g})$. So,
$\Phi ^T((exp_G(s\eta),0_{\frak g}),\psi(Ad^G_{g^{-1}}\xi(x))$ is
just the integral curve of $\phi ^T(\eta, 0_{\frak
g})=(\psi(\eta))^c$ at the point $\psi(Ad^G_{g^{-1}}\xi)(x)\in
A_x.$ In consequence,
\[
\Phi ^T((exp_G(s\eta),0_{\frak
g}),\psi(Ad^G_{g^{-1}}\xi)(x))=\Phi({exp_G(s\eta)},\psi(Ad^G_{g^{-1}}\xi)(x)).
\]
In particular, when $s=1$, we obtain (\ref{l1}).

Furthermore,
\begin{equation}\label{l2}
\psi(Ad^G_{g^{-1}}\xi)(x)=\Phi ^T((e,Ad^G_{g^{-1}}\xi),0_x)
\end{equation}
where $0_x$ is the zero of $A_x$ and $e$ is the identity element
of $G$.

In fact, in order to prove (\ref{l2}), we consider the
one-parameter subgroup
\begin{equation}\label{pi'}
\pi':\R\to G\times {\frak g},\;\;\;
s\mapsto\pi'(s)=(e,sAd^G_{g^{-1}}\xi).
\end{equation}
Then, $\frac{d\pi'}{ds}_{|s=0}=(0_{\frak g},Ad^G_{g^{-1}}\xi)\in
{\frak g}\times {\frak g}$ and, therefore, $\Phi
^T((e,sAd^G_{g^{-1}}\xi),0_x))$ is just the integral curve of
$\psi^{cv} (0_{\frak
g},Ad^G_{g^{-1}}\xi)=(\psi(Ad^G_{g^{-1}}\xi))^v\in {\frak X}(A)$
at the point $0_x\in A_x$, i.e.,
\[
\Phi ^T((e,sAd^G_{g^{-1}}\xi),0_x)=s\psi(Ad^G_{g^{-1}}\xi)(x).
\]
In particular, if $s=1$ we obtain (\ref{l2}).

Now, from (\ref{A0}), (\ref{l1}) and (\ref{l2}), we deduce that
\[
\begin{array}{rcl}
\Phi_g(\psi(Ad^G_{g^{-1}}\xi)(x))&=&\Phi ^T((g,0_{\frak g}),\Phi
^T((e,Ad^G_{g^{-1}}\xi),0_x))\\&=& \Phi ^T((g,0_{\frak g})\cdot
(e,Ad^G_{g^{-1}}\xi),0_x)\\&=&\Phi ^T((e,\xi)\cdot (g,0_{\frak
g}),0_x)= \Phi ^T((e,\xi),\Phi ^T((g,0_{\frak g}),0_{x})).
\end{array}
\]
On the other hand, using \eqref{l1} (with $\xi=0_{\frak g}$), it
follows that $\Phi ^T ((g,0_{\frak g}),0_x)=0_{\phi _g(x)}$. In
addition, from \eqref{l2} (with $g=e$), we obtain that $\Phi ^T
((e,\xi),0_{\phi _g(x)})=\psi (\xi )(\phi _g(x))$. This proves
\eqref{eqlemma} for $g=\exp _G(\eta )$. Finally, using that $G$ is
connected, we conclude that \eqref{eqlemma} holds for all $g\in
G$.

\bigskip
{\it Second step:} Now, we suppose that $G$ is a connected Lie
group with Lie algebra ${\frak g}.$ Denote by $\widetilde{G}$ the
universal covering of $G$ and by $\widetilde{\frak g}$ its
corresponding Lie algebra. Then, the covering projection
$p:\widetilde{G}\to G$ is a local isomorphism of Lie groups and
the map
\[
\widetilde{\Phi}:\widetilde{G}\times A\to A,\;\;\;
\widetilde{\Phi}(\widetilde{g},a_x)=\Phi(p(\widetilde{g}),a_x)
\]
is an action of $\widetilde{G}$ over $A$ by complete lifts with respect to the Lie algebra anti-morphism
 $\widetilde{\psi}=\psi\circ T_ep:\widetilde{\frak g}\to \Gamma(A).$  So,
for all $g\in G$, $x\in M$ and $\xi\in {\frak g}$ there are
$\widetilde{g}\in \widetilde{G}$ and
$\widetilde{\xi}\in\widetilde{\frak g}$ such that
\[
p(\widetilde{g})=g \mbox{ and }
(T_{\widetilde{e}}p)(\widetilde{\xi})=\xi.
\]

Here, $\widetilde{e}$ is the identity element of $\widetilde{G}$.
Therefore, using the first step 
\[
\widetilde{\Phi}_{\widetilde{g}}(\widetilde{\psi}(Ad^{\widetilde{G}}_{\widetilde{g}^{-1}}\widetilde{\xi})(x))=
{\psi}({\xi})({\phi}_{{g}}(x)).\] Finally, since 
$(T_{\widetilde{e}}p)(Ad^{\widetilde{G}}_{\widetilde{g}^{-1}}\widetilde{\xi})=Ad^G_{g^{-1}}\xi$,
then we obtain (\ref{eqlemma}).
\end{proof}

As we previously claimed, from an action of $G$ on $A$ by complete
lifts, we can define an affine action of $TG$ on $A$ as it is
described in the following theorem.
\begin{theorem}\label{induced-TG-action}
Let $\Phi:G\times A\to A$ be an action of the  connected Lie group $G$ by
complete lifts on the Lie algebroid $A$ with respect to
$\psi:{\frak g}\to \Gamma(A).$ Then,
\begin{equation}\label{TG-global}
\Phi^T:(G\times {\frak g})\times A\to A,\;\;\;
\Phi^T((g,\xi),a_x)=\Phi_g(a_x + \psi(\xi)(x))
\end{equation}
defines an affine action of $TG\cong G\times {\frak g}$ over $A$.
Moreover, if $(\xi,\eta)\in {\frak g}\times {\frak g},$ its 
infinitesimal generator $(\xi,\eta)_A$  with
respect to the action $\Phi^T$ is
\begin{equation}\label{TG-infinitesimal}
(\xi,\eta)_A=(\psi(\xi))^c + (\psi(\eta))^v.
\end{equation}
\end{theorem}
\begin{proof}

Equation (\ref{eqlemma}) allows to prove that $\Phi^T:(G\times
{\frak g})\times A\to A$ is an affine action of $TG\cong G\times
{\frak g}$ over $A$. In fact,
\[
\begin{array}{rcl}
\Phi^T((g,\xi)\cdot(h,\eta),a_x)&=&\Phi^T((g\cdot h,\eta +
Ad^G_{h^{-1}}\xi),a_x)=\Phi_{gh}(a_x) + \Phi_{gh}(\psi(\eta +
Ad^G_{h^{-1}}\xi)(x))\\&=&\Phi_g(\Phi_h(a_x) +
\Phi_h(\psi(\eta)(x) + \psi(Ad^G_{h^{-1}}\xi)(x)))\\&=&
\Phi_g(\Phi_h(a_x) + \Phi_h(\psi(\eta)(x))) +
\Phi_g(\psi(\xi)(\phi_h(x)))\\&=&
\Phi^T((g,\xi),\Phi^T((h,\eta),a_x)) \end{array}
\]
and
\[
\Phi^T((e,0_{\frak g}),a_x)=\Phi_e(a_x) + \Phi_e(\psi(0_{\frak
g})(x))=a_x.
\]
Moreover, using the one-parameter subgroups defined in (\ref{l3})
and (\ref{pi'}), one may conclude easily that the infinitesimal
generator $(\xi,\eta)_A$ of $(\xi,\eta)\in {\frak g}\times {\frak
g}$ with respect $\Phi^T$ is
\[ (\xi,\eta)_A=(\xi,0_{\frak g})_A + (0_{\frak
g},\eta)_A=(\psi(\xi))^c + (\psi(\eta))^v.\]
\end{proof}

Note that, under  the same  hypotheses as in Theorem \ref{induced-TG-action},  the action $\Phi:G\times A\to A$ is free and proper if and only if the corresponding action $\phi:G\times M\to M$ on $M$  is free and proper. Moreover, if $\Phi:G\times A\to A$ is free and proper then so is $\Phi^T:TG\times A\to A$. 

On the other hand, we recall that the space of orbits $N/H$ of a free and proper action of a Lie group $H$ on a manifold $N$ is a quotient differentiable manifold and the canonical projection $\pi:N\to N/H$ is a surjective submersion (see \cite{AbMa87}). With this, we prove a preliminary reduction result.

\begin{theorem}\label{reduction}
Let $(A,\lcf\cdot,\cdot\rcf,\rho)$ be a Lie algebroid on $M$ and $\Phi:G\times A\to A$ a 
free and proper action of a connected Lie group $G$ on $A$ by complete lifts with respect to the Lie algebra anti-morphism 
$\psi:{\mathfrak g}\to \Gamma(A).$ Then, $A/TG$ is a Lie algebroid over $M/G$ and the projection $\widetilde{\pi}:A\to A/TG$ is a Lie algebroid epimorphism. 
\end{theorem}
\begin{proof}
Since that $\Phi:G\times A\to A$ is an action by Lie algebroid automorphisms then $A/G$ is a Lie algebroid over $M/G$ with vector bundle projection $\tau/G: A/G\to M/G$ (see \cite{IMMMP, Ma}). The space of sections of this vector bundle may be identified with the one of $G$-invariant sections $\Gamma(A)^G$ of $A$. Under this identification the bracket and the anchor map of the Lie algebroid structure on $A/G$ is just 
$$\lcf X, Y\rcf_{A/G}= \lcf X, Y\rcf, \;\;\;\;\; \rho_{A/G}(X({\pi}(x)))=T_x{\pi}(\rho(X(x))$$
for all $X,Y\in \Gamma(A)^G$ and $x\in M$,  where ${\pi}:M\to M/G$ is the quotient projection corresponding to the induced action $\phi:G\times M\to M.$  

On the other hand,  using  that $\phi:G\times M\to M$  is free, then  we have that  $\psi_x:{\frak g}\to A_x,$ is injective (see Remark \ref{phi}). Thus,  for
all $x\in M,$ we have that 
\[
\dim \{\psi_x(\xi)/\xi\in {\frak g}\}=\dim {\frak
g},\mbox{ for all }x\in M.
\]
Therefore, since $\psi:{\frak g}\to \Gamma(A)$ is a
Lie algebra anti-morphism, we deduce that
\[
\psi({\frak g})=\bigcup_{x\in M}\{\psi_x(\xi)/\xi\in {\frak g}\}
\]
is a Lie subalgebroid of $A$ over $M$.  Moreover,  from Proposition \ref{prop3.2}, we
have that the Lie group $G$ acts by Lie algebroid
automorphisms on  $\psi(\frak g)$. So,  one
may  induce a Lie algebroid structure on the quotient vector
bundle $\psi (\frak g )/G$ such that  $\psi  (\frak g )/G$ is a Lie subalgebroid of
$A /G$. Now, we will show that it is also an ideal.

\medskip If $X\in \Gamma(A)$ is $G$-invariant then $X\circ
\phi_g=\Phi_g\circ X$,  for all $g\in G$. Thus, the flow
$\Upsilon _t(\xi):A \to A$ of the vector field $(\psi 
(\xi))^c$ and the flow $\varphi _t(\xi):M\to
M$ of $(\rho \circ \psi )(\xi)$ satisfy the following property 
\[
\Upsilon _t(\xi)\circ X=X\circ \varphi _t(\xi),\mbox{ for all
 }t\in \R.
\]
Equivalently,
\[
\widehat{X}\circ \Upsilon _t(\xi)^*=\widehat{X}
\]
where  $\Upsilon _t(\xi)^*$ is the dual morphism of $\Upsilon _t(\xi):
A \to A$ and $\widehat{X}$ is the linear function associated
with the section $X$.

Therefore, we have that
\[
\frac{d}{dt}_{|t=0}(\widehat{X}\circ \Upsilon _t(\xi)^*)=0.
\]
Since $\Upsilon _t(\xi)^*$ is the flow of the vector field
$\psi(\xi)^{*c}$, the previous equation is equivalent to the
relation
\begin{equation}\label{corchete0}
\lcf\psi(\xi),X\rcf =0,\;\;\mbox{for all }\xi\in {\frak
g}.
\end{equation}
On the other hand, let $Y$ be a $G$-invariant section of
$\psi({\frak g})$ and $\{\xi_i\}$  a basis of ${\frak
g}$. Then,
\[
Y=\sum Y^i \psi  (\xi _i),
\]
with $Y^i$  real  functions on $A.$ Moreover, using  Proposition \ref{prop3.2} and the fact that $\psi_x$ is injective, we have that
\begin{equation}\label{*}
Y^i\circ \phi _g = (Ad^G)_j^i(g)Y^j,
\end{equation}
where $Ad^G_g\xi _i=(Ad^G)_i^j(g)\xi _j$. Hence, from
\eqref{corchete0},
\[
\lcf X,Y\rcf = \sum \rho (X)(Y^i)\psi (\xi _i).
\]
Now, using that $X$ and $Y$ are $G$-invariant sections, we
obtain that $\lcf X,Y\rcf$ is a $G$-invariant section and, as
a consequence, $\lcf X, Y\rcf$  is a $G$-invariant section of
the vector bundle $\psi (\frak g )\to M$. Thus,
$\psi (\frak g )/G$ is indeed an ideal of
$A/G$ of constant rank (since the action $\Phi$ is free)  and therefore, the quotient vector bundle
$(A/G)/(\psi({\frak g})/G)$ admits a Lie
algebroid structure over $M/G.$

\medskip
Finally, we have  that this vector bundle is isomorphic to $A/TG$ and, thus, a Lie algebroid structure on $A/TG$ is induced in such a
way that this isomorphism is a Lie algebroid isomorphism. In fact,  using (\ref{eqlemma}), one may prove that  $\Phi$ induces a free and proper action $\bar{\Phi}:G\times A/\psi({\mathfrak g})\to A/\psi({\mathfrak g})$   on $A/\psi({\mathfrak g})$ such that the projection $\widetilde{\Psi}_1:A\to A/\psi({\mathfrak g})$ is equivariant, with respect to the action $\Phi^T$ and $\bar{\Phi}$. So,  $\widetilde{\Psi}_1$ induces  a smooth  map $\Psi_1:A/TG\to (A/\psi({\mathfrak g}))\big/G$. Moreover, one easily proves that $\Psi_1$ is a one-to-one correspondence.  On the other hand,  the map 
$$\Psi_2:(A/\psi({\mathfrak g}))\big/G\to (A/G)\big/ (\psi({\mathfrak g})/G),\;\;\, [\widetilde{\Psi}_1(a)]\to \widetilde{\Psi}_2([a])$$ is bijective, where $\widetilde{\Psi}_2:A/G\to (A/G)\big/( \psi({\mathfrak g})/G)$ is the corresponding quotient map. Consequently the reduced vector bundle $A/TG\to M/G$ is isomorphic to the vector bundles 
\begin{equation}\label{tresformas}
(A/\psi({\mathfrak g}))\big/G\to M/G\,\,\,\;\mbox{ and }\;\;\;\; (A/G)\big/ (\psi({\mathfrak g})/G)\to M/G
\end{equation}

Note that, using the above isomorphisms,  the space of sections of the vector bundle $A/TG\to M/G$ may be identified   with the quotient space 
$$\Gamma(A)^G/\Gamma(\psi({\frak g}))^G$$
where $\Gamma (A)^{G}$ (respectively,
$\Gamma ( \psi (\frak g ))^{G}$) is the space of
$G$-invariant sections on $A$ (respectively, $\psi({\frak
g})$). Under this identification the Lie algebroid structure $(\lcf \cdot,
\cdot \rcf _{A/TG},\rho _{A/TG})$ on $A/TG$ is characterized by
\begin{equation}\label{ca}
\begin{array}{l}
\lcf [X],[Y]\rcf _{A/TG} = [\lcf X, Y\rcf ],\\[8pt]
\rho _{A/TG }([X])\circ {\pi}=T{\pi} \circ \rho (X),\quad \mbox{ for
  }  X, Y\in \Gamma (A)^{G}.
\end{array}
\end{equation}
This implies that the canonical projection $\widetilde\pi:A\to A/TG$ is a Lie algebroid epimorphism.

\end{proof}

\begin{examples}{\rm 

$(i)$ In the case when $A = TM$ and $\Phi=T\phi$ is the tangent lift of the action $\phi:G\times M\to M$, the reduced Lie algebroid $TM/TG$ from the previous theorem is isomorphic to $T(M/G)$ with  its standard Lie algebroid structure. 

\medskip 
$(ii)$ For  the case $A={\mathfrak g}\times TM$ from $(ii)$ in Examples \ref{examples}, we obtain that the reduced Lie algebroid  $({\mathfrak g}\times TM)/TG$ with respect to the action 
$\Phi^T: (G\times {\mathfrak g})\times ({\mathfrak g}\times TM)\to {\mathfrak g}\times TM$ given by 
$$\Phi^T((g,\bar{\xi}), (\xi,v_x))=(Ad^G_{g}(\xi - \bar{\xi}), T_x\phi_g(v_x+\bar\xi_M(x)))$$
can be identified with the Atiyah Lie algebroid $TM/G$ induced by the principal bundle $\pi:M\to M/G$. 

We recall the construction of this  last Lie algebroid. Firstly, we denote by $\tau:TM\to M$ the projection of $TM$ on $M$ which is equivariant with respect to the tangent lift action $T\phi:G\times TM\to TM$ and  $\phi:G\times M\to M.$ The sections of the induced vector bundle $\tau/G:TM/G\to M/G$ can be identified with the $G$-invariant vector fields on $M$. Moreover,  the set of $G$-invariant vector fields  is closed with respect to the Lie bracket of vector fields. Using this fact, one can define the Lie algebroid structure  $(\lcf\cdot,\cdot\rcf_{TM/G}, \rho_{TM/G})$  on $TM/G$ by 
$$\lcf X,Y\rcf_{TM/G}=[X,Y],\;\,\; \rho_{TM/G}(X(x))=T_x\pi(X(x)),$$
for $X,Y$ $G$-invariant vector fields of $M$ and $x\in M.$ The corresponding Lie algebroid is known as {\it Atiyah Lie algebroid}  (see \cite{LMM}).

Now, we have  the following vector bundle epimorphism 
\begin{equation}\label{F}F:({\mathfrak g}\times TM)/TG\to TM/G,\;\;\; F([(\xi,v_x)]_{TG})=[v_x+\xi_M]_G.\end{equation}
Note that $F$ is well-defined. Indeed, for all $\xi,\bar{\xi}\in {\mathfrak g}$, $g\in G$ and $v_x\in T_xM$ one has that 
$$
\begin{array}{rcl}
F([Ad_g^G(\xi - \bar{\xi}), T_x\phi_g(v_x + \bar{\xi}_M(x))]_{TG})&=
&[T_x\phi_g(v_x + \bar{\xi}_M(x))+(Ad_g^G(\xi - \bar{\xi}))_M(\phi_g(x))]_G\\
&=&[T_x\phi_g(v_x + \bar{\xi}_M(x))+T_x\phi_g(\xi_M(x) - \bar{\xi}_M(x))]_G\\
&=&[v_x+\xi_M(x)]_G.
\end{array}$$
Moreover, since
\begin{equation}\label{xi=0}
[(\xi,v_x)]_{TG}=[(0,v_x+\xi_M(x))]_{TG},\;\;\; \mbox{ for all } v_x\in T_xM \mbox{ and } \xi\in {\mathfrak g},
\end{equation} 
we deduce that $F$ is a vector bundle isomorphism. 

On the other hand, using (\ref{xi=0}), we obtain that if $\{X_i\}$ is a local basis of $G$-invariant vector fields on $M$ then $\{[(0,X_i)]_{TG}\}$ is a local base  of  $\Gamma(({\mathfrak g}\times TM)/TG).$  This fact allows to prove  that $F$ is a Lie algebroid isomorphism.  

}\end{examples}

\subsection{Reduction of Symplectic Lie algebroids}\label{section3.2}
A Lie
algebroid $(A,\lcf\cdot,\cdot\rcf,\rho)$ on the manifold $M$ is {\it symplectic-like } if
there is a nondegenerate $2$-section $\Omega\in
\Gamma(\wedge^2A^*)$ on $A^*$ which is closed, i.e.
$d^A\Omega=0.$ In such a case, for each function $f:M\to {\mathbb R}$ on $M$, we have the {\it Hamiltonian section } on $A$ which is characterized  by
$$i_{{\mathcal H}^\Omega_f}\Omega=d^Af.$$

The base space $M$ of a symplectic-like Lie algebroid $A$ is a Poisson manifold, where the Poisson bracket on $M$ is given by
\begin{equation}\label{Pois}
\{f,g\}=\rho({\mathcal H}^\Omega_g)(f)\;\;\;f,g\in C^\infty(M).
\end{equation}
(see \cite{LMM, K, MX}). 

Note that if $f\in C^\infty(M)$ then the Hamiltonian vector field of $f$ with respect to the Poisson structure on $M$ is $\rho({\mathcal H}_f^\Omega)$. Thus, the solutions of Hamilton's equations for $f$ are the integral curves of the vector field $\rho({\mathcal H}_f^\Omega).$

Now, in the rest of this section, we suppose that $(A,\lcf\cdot,\cdot\rcf,\rho,\Omega)$ is a
symplectic-like Lie algebroid over $M$, that $\Phi:G\times A\to A$ is an  action of a connected
Lie group $G$ on  $A$ by complete lifts with respect to the Lie algebra anti-morphism $\psi:{\mathfrak g}\to \Gamma(A)$ and that 
$J:M\to {\frak g}^*$ is an equivariant smooth map, i.e., 
$$
Coad^G_g(J(x))=J(\phi_g(x)),\;\;\; \forall x\in M,\;\;\;\forall
g\in G.$$

 The action $\Phi$ is said to be a {\it Hamiltonian action with momentum map } $J:M\to {\frak g}^*$ if
\begin{equation}\label{3,18'}\Phi_g^*(\Omega)=\Omega \mbox{ and } i_{\psi(\xi)}\Omega=d^A {J}_\xi, \mbox{ for all } g\in G \mbox{ and } \xi\in {\mathfrak g},\end{equation}
where ${J}_\xi$ is the real function on $M$ given by
\begin{equation}\label{e2}
{J}_\xi(x)=J(x)(\xi), \mbox{ for  } x\in M.
\end{equation}

Note that the previous condition implies that 
\begin{equation}\label{3.19'}
{\mathcal L}_{\psi(\xi)}^A\Omega=0.
\end{equation}

Now, let  $J^T:A\to ({\frak g}\times {\frak g})^*\cong
{\frak g}^*\times {\frak g}^*$ be the map given by
\begin{equation}\label{JA}
J^T(a)=((TJ\circ \rho )(a) , J(\tau(a))).
\end{equation}

Then we have

\begin{lemma}\label{induced-Poisson-action}
The map $J^T:A\to {\mathfrak g}^*\times {\mathfrak g}^*$ is equivariant for the action $\Phi^T:TG\times A\to A$.
\end{lemma}

\begin{proof}
Let $(g,\xi )\in G\times \frak g\cong TG$ and $a\in A_x$. Since
$J$ is equivariant, we have that
\begin{equation}\label{lema-equiv1}
T_xJ (\rho ( \psi (\xi )(x)))=coad^G_\xi (J(x)), \quad \mbox{for
}\xi\in \frak g\mbox{ and } x\in M.
\end{equation}

Moreover, using that $\Phi _g$ is a Lie algebroid morphism over
$\phi _g$ and that $J$ is equivariant, we obtain
\begin{equation}\label{lema-equiv2}
TJ\circ \rho \circ \Phi _g = TJ \circ T\phi _g\circ \rho = Coad^G
_g\circ TJ\circ \rho.
\end{equation}
As a consequence, from (\ref{coadjunto-TG}), (\ref{JA}),
(\ref{lema-equiv1}) and (\ref{lema-equiv2}), we have that 
\begin{eqnarray*}
Coad ^{TG}_{(g,\xi )} (J^T(a))&=& Coad ^{TG}_{(g,\xi )} (TJ (\rho
(a)) ,J(\tau (a)))\\&=&(Coad^G_g(TJ\circ
\rho(a))+Coad^G_g(coad^G_\xi J(\tau (a))),Coad^G_g (J(\tau (a))))\\
&=& (Coad^G_g(TJ\circ \rho(a))+Coad^G_g(coad^G_\xi
J(\tau (a))),J (\phi _g(\tau (a))))\\
&=& ((TJ \circ \rho)( \Phi _g(a) ))+ (TJ\circ \rho)(\Phi _g(\psi
(\xi)(\tau(a)))), J (\tau (\Phi ^T _{(g,\xi )}(a))) )\\
&=& ((TJ\circ \rho )(\Phi ^T_{(g,\xi )}(a)),J(\tau(\Phi ^T_{(g,\xi
)}(a))))\\
&=& J^T (\Phi ^T_{(g,\xi )}(a)).
\end{eqnarray*}
Hence $J^T$ is equivariant with respect to $\Phi^T.$
\end{proof}

\begin{proposition}\label{prereduction}
Let  $(A,\lcf\cdot,\cdot\rcf,\rho,\Omega)$  be a symplectic-like Lie algebroid over $M$,
 $\Phi:G\times A\to A$  a Hamiltonian action of a connected Lie group $G$ on $A$ 
 with Lie algebra anti-morphism $\psi:{\mathfrak g} \to \Gamma(A)$  and  equivariant momentum map $J:M\to \frak g^*$. Let 
$\mu\in {\frak g}^*$ be a  regular value of $J$ such that  $T_x J\circ \rho : A_x \to T_{\mu}{\frak g}^*$ has  constant rank for all  $x\in J^{-1}(\mu).$ Then, 
\begin{enumerate}
\item $\pl$ is a Lie subalgebroid of $A$ over $J^{-1}(\mu).$
\item The restriction $\psi_\mu:{\mathfrak g}_\mu\to \Gamma(A)$ of $\psi$ to the isotropy algebra
${\frak g}_\mu$ of $\mu$ with respect to the coadjoint action
takes values in $\Gamma(\pl)$.
\item The isotropy Lie group $G_\mu$ of $\mu$ with respect to the
coadjoint action acts on $\pl$ by complete lifts with respect to
$\psi_\mu:{\frak g}_\mu\to \Gamma(\pl).$
\item The action of $G_\mu$ on the Lie subalgebroid $\pl$ induces
an affine action $\Phi^T_\mu:TG_\mu\times (J^T)^{-1}(0,0)\to (J^T)^{-1}(0,0).$
\end{enumerate}
\end{proposition}

\begin{proof}
(i) Note that since $\mu$ is a regular value of $J,$
$J^{-1}(\mu)$ is a regular submanifold of $M.$ In fact, $(0,\mu)$ is
a regular value for $J^T:A\to \mathfrak g^*\times \mathfrak
g^*\cong T{\frak g}^*.$ Thus, using that $T_xJ\circ \rho_x:A_x\to T_\mu{\mathfrak g}^*$ has constant rank for all $x\in J^{-1}(\mu),$  we deduce that
\[
(J^T)^{-1} (0,\mu)=\{a\in A/TJ(\rho(a))=0,\;\; J(\tau(a))=\mu\}
\]
is a vector subbundle of $A$ on  $J^{-1}(\mu)$   of rank $n-\dim G,$ where $n=\mbox{rank } A.$ 
  On the other hand, from (\ref{JA}) we have that $J^T$ is a Lie algebroid morphism then it is  straightforward to see that  the restriction $\tau_{|(J^T)^{-1} (0,\mu)}: (J^T)^{-1} (0,\mu)\to J^{-1}(\mu)$ of $\tau:A\to M$ to $(J^T)^{-1} (0,\mu)$ is a 
Lie subalgebroid of $(A,\lcf\cdot,\cdot\rcf,\rho)$.

\medskip

(ii)  If $x\in J^{-1}(\mu)$ and
\[
\xi\in {\frak g}_\mu=\{\eta\in {\frak g}/coad^G_\eta\mu=0\}
\]
then, since  $J$ is an equivariant map, we have that
\[
T_xJ(\rho_x (\psi (\xi )(x))) = T_xJ(\xi_M(x))= coad^G _\xi(\mu) =
0.
\]
Thus, the restriction of $\psi(\xi)$ to $J^{-1}(\mu)$ is a section
of the vector bundle $\pl\to J^{-1}(\mu)$.

\medskip

(iii) Using the equivariance of $J$ and the fact that $\Phi_g$
is a Lie algebroid automorphism, for any $g\in G$ we have that
the action $\Phi:G\times A\to A$ induces an action $\Phi _\mu$ of
$G_\mu$ on $\pl$. In fact, if $g\in G_\mu$, $a\in \pl$ and $x\in J^{-1}(\mu)$, then
$$J(\phi_g(x))=Coad_g^G(J(x))=Coad_g^G\mu=\mu$$
and 
\begin{eqnarray*}
(TJ\circ \rho )(\Phi _g(a))&=&
T(J\circ \phi _g)(\rho (a))= T(Coad^G _{g}\circ J)(\rho
(a))= 0.
\end{eqnarray*}
Moreover, from  (ii), we have that $\Phi _\mu$ is an action by
complete lifts with respect to the Lie algebra anti-morphism
$\psi_\mu.$

\medskip

(iv) It is a direct consequence of (iii) and Theorem
\ref{induced-TG-action}. \end{proof}

Let
$\mu\in {\frak g}^*$ be  a regular value of $J:M\to {\frak g}^*$ such that  $T_x J\circ \rho_x : A_x \to T_{\mu}{\frak g}^*$ has constant rank for all  $x\in J^{-1}(\mu).$   Suppose that the corresponding action $\phi_\mu:G_\mu\times J^{-1}(\mu)\to J^{-1}(\mu)$ is free and proper. Then, using Theorem \ref{reduction} and Proposition \ref{prereduction}, we obtain that $A_\mu=(J^T)^{-1}(0,\mu)/TG_\mu$ is a Lie algebroid over  $J^{-1}(\mu)/G_\mu.$
In the following result, we will
prove that $A_\mu$ is a symplectic-like Lie algebroid.  For this purpose, we will need  the following properties.

\begin{lemma}\label{lemma2}
Let $(A,\lcf\cdot,\cdot\rcf,\rho,\Omega)$ be a symplectic-like Lie algebroid
over the manifold $M$ and $\Phi:G\times A\to A$  a Hamiltonian
action of a connected Lie group $G$ on $A$ with equivariant momentum map $J:M\to
{\frak g}^*$ and associated Lie algebra anti-morphism $\psi:{\frak g}\to
\Gamma(A).$ If $\mu\in {\frak g}^*$, then for any $x\in M$, 
\begin{enumerate}
\item
$(\psi_\mu)_x({\frak g}_\mu)=\psi_x({\frak g})\cap \ker (T_xJ\circ
\rho_x)$
\item $\ker(T_xJ\circ \rho_x)=(\psi_x({\frak g}))^\perp=\{a_x\in
A_x/\Omega_x(a_x,b_x)=0, \forall b_x\in \psi_x({\frak g})\}.$
\end{enumerate}
\end{lemma}
\begin{proof}
$(i)$ It is an immediate consequence of the fact that $J$ is
equivariant.

$(ii)$ If $a_x\in A_x$, using (\ref{diferencial}) and (\ref{e2}),
we deduce that
\[
\Omega(a_x,\psi(\xi)(x))=-(i_{\psi(\xi)}\Omega)(a_x)=-(d^A
{J}_\xi)(a_x)=-(T_xJ(\rho_x(a_x)))({\xi}),
\]
for all $\xi\in {\frak g}.$ Thus, one concludes immediately (ii)
from this relation. \end{proof}

The following result may be seen as  the analogous of
Marsden-Weinstein reduction Theorem for symplectic-like Lie algebroids.

\begin{theorem}\emph{Reduction Theorem of symplectic-like Lie algebroids}\label{2.8co}
Let $(A,\lcf\cdot,\cdot\rcf,\rho,\Omega)$ be a symplectic-like Lie
algebroid and $\Phi:G\times A\to A$  a Hamiltonian action of a connected Lie group $G$ on $A$ with
equivariant momentum map $J:M\to {\frak g}^*$ and associated Lie algebra anti-morphism
$\psi:{\frak g}\to \Gamma(A)$. Suppose that 
$\mu\in {\frak g}^*$  is a regular value of $J:M\to {\frak g}^*$ such that  $T_x J\circ \rho_x : A_x \to T_{\mu}{\frak g}^*$ has constant rank for all  $x\in J^{-1}(\mu)$ and the restricted action $\phi_\mu:G_\mu\times J^{-1}(\mu)\to J^{-1}(\mu)$ is free and proper. Then
$A_\mu=(J^T)^{-1}(0,\mu)/TG_\mu$ is a symplectic-like Lie algebroid
over $J^{-1}(\mu)/G_\mu$ with symplectic-like section $\Omega_\mu$ characterized by the condition
\[
\widetilde{\pi}_\mu^*\Omega_\mu=\widetilde{\iota}_\mu^*\Omega,
\]
where $\widetilde{\pi}_\mu:(J^T)^{-1} (0,\mu)\to A_\mu$ is the
canonical projection and
$\widetilde{\iota}_\mu:(J^T)^{-1}(0,\mu)\to A$ is the canonical
inclusion.
 \end{theorem}

\begin{proof} Since $A_\mu$ is a Lie algebroid over $J^{-1}(\mu)/G_\mu$ then one needs  to prove that this algebroid is  symplectic-like. 

Let $\widetilde{\Omega}_\mu=\widetilde{\iota}_\mu^*\Omega$ be the $2$-cocycle on the Lie subalgebroid $(J^T)^{-1}(0,\mu)\to J^{-1}(\mu)$ induced by $\Omega.$  We will prove that  $\widetilde{\Omega}_\mu$
induces a symplectic-like $2$-section $\Omega_\mu$ over $A_\mu$. 

Suppose that 
${X_\mu}, {Y}_\mu\in\Gamma(A_\mu).$ Then, we may choose two sections $\widetilde{X}_\mu,\widetilde{Y}_\mu\in \Gamma((J^T)^{-1}(0,\mu))$ such that the following diagram is commutative

\begin{picture}(375,60)(60,40)
\put(190,20){\makebox(0,0){$J^{-1}(\mu)/G_\mu$}} \put(245,25){${X}_\mu,{Y}_\mu$}
\put(215,20){\vector(1,0){90}} \put(315,20){\makebox(0,0){$A_\mu$}}
\put(190,50){$\pi_\mu$} \put(210,70){\vector(0,-1){40}}
\put(320,50){$\widetilde{\pi}_\mu$} \put(310,70){\vector(0,-1){40}}
\put(198,80){\makebox(0,0){$J^{-1}(\mu)$}} \put(245,85){$\widetilde{X}_\mu,\widetilde{Y}_\mu$}
\put(215,80){\vector(1,0){90}} \put(335,80){\makebox(0,0){$(J^T)^{-1}(0,\mu)$}}
\end{picture}

\vspace{1cm}

We will see that $\widetilde{\Omega}_\mu(\widetilde{X}_\mu,\widetilde{Y}_\mu)$ is a $G_\mu$-invariant function (or, equivalently,  a $\pi_\mu$-basic function).

Denote by $(\lcf\cdot,\cdot\rcf_{\pl},\rho_{\pl})$ the Lie
algebroid structure on $\pl\to J^{-1}(\mu).$

As we know, the vertical bundle of $\pi_\mu$ is generated by the vector fields on $J^{-1}(\mu)$ of the form $\rho_{(J^T)^{-1}(0,\mu)}(\psi_\mu(\xi)),$ with $\xi\in {\frak g}_\mu.$

Now, we have that 
\[\begin{array}{rcl}
(\rho_{(J^T)^{-1}(0,\mu)}(\psi_\mu(\xi)))(\widetilde{\Omega}_\mu(\widetilde{X}_\mu,\widetilde{Y}_\mu))&\kern-10pt=&\kern-10pt({\mathcal
L}_{\psi_\mu(\xi)}^{\pl}\widetilde{\Omega}_\mu)(\widetilde{X}_\mu,\widetilde{Y}_\mu)
+ \widetilde{\Omega}_\mu(\lcf\psi_\mu(\xi),\widetilde{X}_\mu\rcf_{\pl},\widetilde{Y}_\mu)\\[8pt]
&&+
\widetilde{\Omega}_\mu(\widetilde{X}_\mu,\lcf\psi_\mu(\xi),\widetilde{Y}_\mu\rcf_{\pl}).
\end{array}\]

On the other hand, using that $\widetilde{X}_\mu$ and $\widetilde{Y}_\mu$ are $G_\mu$-invariant, we deduce that 
$$\lcf\psi_\mu(\xi),\widetilde{X}_\mu\rcf_{\pl}=\lcf
\psi_\mu(\xi),\widetilde{Y}_\mu\rcf_{\pl}=0$$ (see (\ref{corchete0})). In addition, from (\ref{3.19'}), it follows that 
\[
{\mathcal L}_{\psi_\mu(\xi)}^{\pl}\widetilde{\Omega}_\mu=0,
\]
which implies that 
\[
\rho_{(J^{T})^{-1}(0,\mu)}(\psi_\mu(\xi))(\widetilde{\Omega}_\mu(\widetilde{X}_\mu,\widetilde{Y}_\mu))=0.
\]
 Thus, for all
${X}_\mu, {Y}_\mu\in \Gamma(A_\mu)$ there is a function
${\Omega}_\mu({X}_\mu, {Y}_\mu)$ on $J^{-1}(\mu)/G_\mu$  such that
\[
\Omega_\mu(X_\mu,Y_\mu)\circ
\pi_\mu=\widetilde{\Omega}_\mu(\widetilde{X}_\mu,\widetilde{Y}_\mu).
\]
Note that the function ${\Omega}_\mu({X}_\mu, {Y}_\mu)$ does not depend on the choosen sections $\widetilde{X}_\mu,\widetilde{Y}_\mu\in \Gamma(\pl)$ which project on $\widetilde{X}_\mu$ and $\widetilde{Y}_\mu,$ respectively. In fact, from Lemma \ref{lemma2}, we have that 
$$\ker(\widetilde{\Omega}_\mu(x))=(\ker\widetilde{\pi}_\mu)_{|(J^{T})_x^{-1}(0,\mu)}=(\psi_\mu)_x({\frak g}_\mu) \mbox{ for all } x\in
J^{-1}(\mu).$$
Therefore, the map
\[\Gamma(A_\mu)\times \Gamma(A_\mu)\to C^\infty(J^{-1}(\mu)/G_\mu),\;\;\; ({X}_\mu, {Y}_\mu)\mapsto {\Omega}_\mu({X}_\mu, {Y}_\mu)\]
defines a section ${\Omega}_\mu$ of the vector bundle $\Lambda^2A_\mu^*\to J^{-1}(\mu)/G_\mu$ and it is clear that 
$$\widetilde{\pi}_\mu^*{\Omega_\mu}=\widetilde{\iota}_\mu^*\Omega=\widetilde{\Omega}_\mu.$$
This implies that 
$$\widetilde{\pi}_\mu^*(d^{A_\mu}{\Omega_\mu})=d^{\pl}(\widetilde{\pi}_\mu^*{\Omega}_\mu)=d^{\pl}\widetilde{\iota}_\mu^*\Omega=\widetilde{\iota}_\mu^*(d^A\Omega)=0$$
and, since $\widetilde{\pi}_\mu:\pl\to A_\mu=\pl/TG_\mu$ is an epimorphism of vector bundles, we conclude that $d^{A_\mu}{\Omega}_\mu=0.$ 

Finally, we will prove that $\Omega_\mu$ is non-degenerate. In
fact, if $x\in J^{-1}(\mu)$ and $\widetilde{v}_x\in \ker
(T_xJ\circ \rho_x)$ is such that
\[
{\Omega}_\mu({\pi_\mu(x)})(\widetilde{\pi}_\mu(\widetilde{v}_x),u_{\pi_\mu(x)}))=0,\;\;\;
\forall u_{\pi_\mu(x)}\in (A_\mu)_{\pi_\mu(x)}
\]
then
\[
\widetilde{\Omega}(x)(\widetilde{v}_x,\widetilde{u}_x)=0 \;\;\mbox{
for all }\widetilde{u}_x\in \ker (T_xJ\circ \rho_x).
\]
Consequently  (see Lemma \ref{lemma2})
\[
\widetilde{v}_x\in (\ker (T_xJ\circ \rho_x))^\perp=\psi_x({\frak
g})\]
and thus, 
\[
\widetilde{v}_x\in (\psi_\mu)_x({\mathfrak g}_\mu)
\]
which implies that $\widetilde{\pi}_\mu(\widetilde{v}_x)=0$.

\end{proof}

\begin{remark}{\rm In the particular case when $M$ is a symplectic manifold and $A$ is the standard symplectic-like Lie algebroid $TM\to M$ then Theorem \ref{2.8co} reproduces the classical Marsden-Weinstein reduction result for the symplectic manifold $M.$}
\end{remark}

Since $A_\mu\to J^{-1}(\mu)/G_\mu$ is a symplectic-like Lie algebroid, the base space $J^{-1}(\mu)/G_\mu$ is a Poisson manifold. In fact, we will prove that $J^{-1}(\mu)/G_\mu$ is the reduced Poisson manifold $(M,\{\cdot,\cdot\})$ obtained from the reduction process of Marsden-Ratiu \cite{MR}.

\begin{theorem}\label{T2.10}
Under the hypotheses of Theorem \ref{2.8co}, if $\{\cdot,\cdot\}_\mu$ is the Poisson bracket on $J^{-1}(\mu)/G_\mu$, we have that
\begin{equation}\label{MarRa}
\{ \tilde{f},\tilde{g}\}_\mu\circ \pi_\mu=\{f,g\}\circ i_\mu
\end{equation}
for $\tilde{f},\tilde{g}\in C^\infty( J^{-1}(\mu)/G_\mu)$, where $i_\mu:J^{-1}(\mu)\to M$ is the canonical inclusion and $f,g\in C^\infty(M)$ are arbitrary $G$-invariant extensions of $\tilde{f}\circ \pi_\mu$ and $\tilde{g}\circ \pi_\mu,$ respectively.
\end{theorem}
\proof{ From Theorem \ref{2.8co}, we deduce that
$(A_\mu,\lcf\cdot,\cdot\rcf_{A_\mu},\rho_{A_\mu},\Omega_\mu)$ is a
symplectic-like Lie algebroid on $J^{-1}(\mu)/G_\mu$. Then one can
define a Poisson structure on $J^{-1}(\mu)/G_\mu$ as in
(\ref{Pois}). We will prove that the associated  Poisson bracket
$\{\cdot\;\cdot\}_\mu$ satisfies (\ref{MarRa}).

If $\tilde{f}, \tilde{g}:J^{-1}(\mu)/G_\mu\to {\Bbb R}$ are two real functions  on $J^{-1}(\mu)/G_\mu$ and   
${f}:M\to \R$, $g:M\to \R$  are arbitrary $G$-invariant extensions of $\tilde{f}\circ \pi_\mu$ and $\tilde{g}\circ \pi_\mu,$ respectively, 
for any 
$\xi\in {\frak g}$ satisfying 
$\rho(\psi(\xi))(f)=\rho(\psi(\xi))(g)=0,$ we have that 
\[
d^Af(\psi(\xi))=d^Ag(\psi(\xi))=0,
\]
or, equivalently,
\[
\Omega({\mathcal H}_f^\Omega,\psi(\xi))=\Omega({\mathcal H}_g^\Omega,\psi(\xi))=0.
\]
Therefore, ${\mathcal H}_f^\Omega(x),{\mathcal H}_g^\Omega(x)\in
\psi_x({\frak g})^\perp=(J^T)^{-1}_x(0,\mu),$ for all $x\in
J^{-1}(\mu).$

On the other hand, if $x\in J^{-1}(\mu)$ and $a_x\in
(J^T)^{-1}_x(0,\mu)$ then, using Theorem \ref{2.8co}, the fact that
$(\widetilde{\pi}_\mu,\pi_\mu)$
 is a Lie algebroid epimorphism and that $(J^T)^{-1}(0,\mu)\to
J^{-1}(\mu)$ is a Lie subalgebroid of $A$, one deduces
that
\[
\begin{array}{rcl}
\Omega_\mu(\widetilde{\pi}_\mu({\mathcal
H}_f^\Omega(x)),\widetilde{\pi}_\mu(a_x))) &=&\Omega({\mathcal
H}_f^{\Omega}(x),a_x)=(d^Af)(a_x)\\&=&
(d^{(J^T)^{-1}(0,\mu)}(\widetilde{f}\circ
\pi_\mu))(a_x)=(d^{A_\mu}\widetilde{f})(\widetilde{\pi}_\mu(a_x))\\&=&
\Omega_\mu({\mathcal
H}_{\widetilde{f}}^{\Omega_\mu}({\pi_\mu(x)}),\widetilde{\pi}_\mu(a_x))
\end{array}
\]

So, since $\Omega_\mu$ is non-degenerate
\[
\widetilde{\pi}_\mu({\mathcal H}_f^\Omega({x}))={\mathcal
H}_{\tilde{f}}^{\Omega_\mu}({\pi_\mu(x)}).
\]
Thus, using again that $(\widetilde{\pi}_\mu,\pi_\mu)$ is an
epimorphism of Lie algebroids, we conclude that
\[
T_x\pi_\mu(\rho_{(J^T)^{-1}(0,\mu)}({\mathcal
H}_f^\Omega(x))=\rho_{A_\mu}({\mathcal
H}_{\widetilde{f}}^{\Omega_\mu}({\pi_\mu(x)})),\;\;\forall x\in
J^{-1}(\mu).
\]
Therefore, if $x\in J^{-1}(\mu)$
\[
\begin{array}{rcl}
\{ \widetilde{f},\widetilde{g}\}_\mu
(\pi_\mu(x))&=&-(\rho_{A_\mu}({\mathcal
H}_{\widetilde{f}}^{\Omega_\mu}({\pi_\mu(x)})))\widetilde{g}
\\&=&-(T_x\pi_\mu(\rho_{(J^T)^{-1}(0,\mu)}({\mathcal
H}_f^\Omega(x))))\widetilde{g}\\&=&
-(\rho_{(J^T)^{-1}(0,\mu)}({\mathcal
H}_f^\Omega(x)))(\widetilde{g}\circ \pi_\mu)\\&=&-(\rho({\mathcal
H}_f^\Omega(x))g=\{f,g\}(x).
\end{array}
\]
\hfill$\Box$}

\section{The canonical cover of a fiberwise linear Poisson structure}\label{section prolongued} 

A standard example of a symplectic manifold is the cotangent bundle $T^*M$ of a manifold $M$ with its canonical symplectic structure. In the setting of Lie algebroids,  the tangent bundle $\pi_{T^*M}:T(T^*M)\to T^*M$ of $T^*M$ is a  symplectic-like Lie algebroid. This symplectic-like Lie algebroid may be considered as the canonical cover of the canonical symplectic structure  on $T^*M.$ In fact, it is a particular case of a type of symplectic-like Lie algebroids, {\it the prolongation of a Lie algebroid $A$  on  its dual $A^*$} in the terminology of  \cite{LMM}, which may be considered as {\it the canonical cover of the  fiberwise linear Poisson structure of $A^*.$}

 In this section we will describe this last canonical cover  and will prove that if a Lie group $G$ acts freely and properly on $A$ by complete lifts then one may introduce an equivariant momentum map with respect to a certain canonical action by complete lifts on its canonical cover.   

\bigskip 

\emph{The vector bundle ${\mathcal T}^AA^*\to A^*$.}
 Let
$(A,\lcf\cdot,\cdot\rcf,\rho)$ be a Lie algebroid of rank $n$ over a
manifold $M$ of dimension $m$ with $\tau:A\to M$ the associated vector bundle projection and  let $F:M'\to M$  be  a smooth map from a manifold $M'$ to $M.$   If $x'\in M',$ we consider the vector subspace
$$({\mathcal T}^AM')_{x'}=\{(a,v)\in A_{F(x')}\times T_{x'}M'/\rho(a)=T_{x'}F(v)\}$$
of $A_{F(x')}\times T_{x'}M'$  of dimension $n+m'-\dim (\rho(A_{F(x')}) + T_{x'}F(T_{x'}M'))$, where $m'$ is the dimension of $M'$. If  we suppose that $\dim (\rho(A_{F(x')}) + T_{x'}F(T_{x'}M'))$ is constant over $F(M')$ (for instance, if $F$ is a submersion) then ${\mathcal T}^AM'$ 
is a vector bundle over $M'$ which is called {\it the prolongation of $A$ over $F$} (see \cite{HiMa, LMM}).
In this case, a section ${\mathcal X}$ of ${\mathcal T}^AM'\to M'$  is said to be {\it projectable } if there exist a section $X$ of $A$ and a vector field $V$ on $M',$  $F$-projectable over $\rho(X),$ satisfying
\[
{\mathcal Z}(m')=(X(F(m')),V(m')), \mbox{ for all }
m'\in M'.
\]
The section ${\mathcal Z}$ will be denoted by
${\mathcal Z}=(X,V).$ Note that one may choose a local basis $\{{\mathcal Z}_I\}$ of $\Gamma({\mathcal T}^AA^*)$
such that, for all $I$, ${\mathcal Z}_I$ is a projectable section.

On the other hand, a  section $\widetilde{\gamma}$ of the dual vector bundle $({\mathcal T}^AM')^*\to M'$ is said to be {\it projectable } if there  exist a section  $\alpha$  of $A^*$ and a $1$-form $\beta$ of $M'$ such that 
$$\widetilde{\gamma}(X,V)=\alpha(X)\circ F + \beta(V),\;\; \mbox{ for $(X,V)$ a projectable section of ${\mathcal T}^AM'$}.$$
In such a case we will use $\widetilde{\gamma}=(\alpha,\beta).$ Note that one may choose a local basis $\{{\mathcal Z}^I\}$ of $\Gamma(({\mathcal T}^AA^*)^*)$ such that, for all I, ${\mathcal Z}^I$ is a projectable section.

A particular case is when the function $F$ is the dual bundle projection $\tau_*:A^*\to M$ of the Lie algebroid $A$. Then the prolongation $\tau_{{\mathcal T}^AA^*}:{\mathcal T}^AA^*\to A^*$ of $A$ over $\tau_*:A^*\to M$ is called the  {\it $A$-tangent bundle of $A^*$}. In such a  case, $\dim (\rho(A_{x}) + T_{\alpha_x}\tau_{*}(T_{\alpha_x}A^*))$ is just the dimension of $M$, for all $\alpha_x\in A_x^*,$ and the rank of ${\mathcal T}^AA^*$ is $2n.$ 

A  basis of local sections of the vector bundle $\tau_{{\mathcal T}^AA^*}:{\mathcal T}^AA^*\to A^*$ is defined as follows. 
If $(x^i)$ are local coordinates on an open
subset $U$ of $M$, $\{e_I\}$ is a basis of sections of the vector
bundle $\tau^{-1}(U) \to U$ and $(x^i, y_{I})$ are the
corresponding local coordinates on $A^*$ then $\{{\mathcal X}_I,
{\mathcal Y}^I\}$ is a local basis of $\Gamma({\mathcal T}^AA^*)$,
where ${\mathcal X}_{I}$ and ${\mathcal Y}^{I}$ are the
projectable sections defined by
\begin{equation}\label{base}
{\mathcal X}_{I} = (e_{I}, \rho_{I}^i \displaystyle
\frac{\partial}{\partial x^i}), \makebox[.5cm]{} {\mathcal Y}^{I}
= (0, \displaystyle \frac{\partial}{\partial y_{I}}).
\end{equation}

\bigskip 

\emph{The symplectic-like Lie algebroid structure on ${\mathcal T}^AA^*\to A^*$.}
The vector bundle ${\mathcal T}^AA^*$ admits a Lie algebroid structure $(\lcf\cdot,\cdot\rcf_{{\mathcal T}^AA^*},\rho_{{\mathcal T}^AA^*})$ which is characterized by the following conditions
\[
\lcf (X,V),(X',V')\rcf_{{\mathcal T}^AA^*}=(\lcf X,Y\rcf,[V,V']),\;\;\;\; \rho_{{\mathcal T}^AA^*}(X,V)=V,
\]
for $(X,V),(X',V')$ projectable sections of ${{\mathcal T}^AA^*}.$

If $d^{{\mathcal T}^AA^*}$ is the differential  associated with this Lie algebroid structure, then  
\begin{equation}\label{dTA}
\begin{array}{rcl}
d^{{\mathcal T}^AA^*}f(X_1,V_1)&=& df(V_1) \\
d^{{\mathcal T}^AA^*}(\alpha,\beta)((X_1,V_1), (X_2,V_2))&=& d^A\alpha(X_1,X_2)\circ \tau_*+ d\beta( V_1,V_2)
\end{array}
\end{equation}
 where   $f:A^*\to {\Bbb R}$ is a smooth function, $(\alpha,\beta)\in \Gamma(({\mathcal T}^AA^*)^*)$ and $(X_i,V_i)\in\Gamma({\mathcal T}^AA^*)$  are  projectable sections  of  $({\mathcal T}^AA^*)^*$ and ${\mathcal T}^AA^*$, respectively.

The canonical section $\lambda_A$ of the dual bundle to ${\mathcal
T}^A A^*$ (which is called  {\it the Liouville section  associated with the Lie algebroid $A$}) may be defined as follows
\begin{equation}\label{Liouville}
\lambda_A(a,v_\alpha)=\alpha(a), \mbox{ for } \alpha\in
A^* \mbox{ and } (a,v_\alpha)\in {{\mathcal T}_\alpha^AA^*}.
\end{equation}

The section $\Omega_A$ of $\wedge^2({{\mathcal T}^AA^*})^*\to A^*$ given by
\[
\Omega_A=-d^{{{\mathcal T}^AA^*}}\lambda_A
\]
is nondegenerate and $d^{{\mathcal
T}^A A^*}\Omega_A=0.$ Thus, $\Omega_A$ is a symplectic-like section of the
Lie algebroid ${{\mathcal T}^AA^*}\to A^*$ which is called 
 {\it the canonical symplectic-like section associated
with the Lie algebroid $A$}. The Poisson structure on the base space $A^*$ induced by this
symplectic-like section  is just the linear Poisson
structure on $A^*$ associated with the Lie algebroid $A$ (see \cite{LMM}). For this reason, ${\mathcal T}^AA^*$ may be considered as {\it the  canonical cover of the fiberwise linear Poisson structure on $A^*$}. 
If $\{{\mathcal X}_{I}, {\mathcal Y}^{I}\}$ is the local basis of sections of ${\mathcal T}^AA^*$ described in  (\ref{base}), 
the local expressions of $\lambda_A$ and $\Omega_A$ are
\[
\lambda_{A} = y_{I}{\mathcal X}^I, \makebox[.5cm]{} \Omega_{A} =
{\mathcal X}^I \wedge {\mathcal Y}_{I} + \frac{1}{2} C_{IJ}^K y_K
{\mathcal X}^I \wedge {\mathcal X}^J
\]
where
 $\{{\mathcal X}^I, {\mathcal Y}_{I}\}$ is the dual basis of $\{{\mathcal X}_{I}, {\mathcal Y}^{I}\}$   and $C_{IJ}^K$ are the local structure functions of the bracket $\lcf\cdot, \cdot\rcf$ (for more details, see \cite{LMM}).

\begin{examples}\label{4.1}{\rm 
$(i)$
Note that if $A$ is the standard Lie algebroid $TM$ then the symplectic-like Lie algebroid 
${\mathcal T}^AA^*\to A^*$ may be identified with the standard Lie algebroid $T(T^*M)\to T^*M$ and, under this identification,  $\Omega_{A}$ is the
canonical symplectic $\Omega_M$ structure of $T^*M$.

\medskip 

$(ii)$
For  the case $A={\mathfrak g}\times TM$ from $(ii)$ in Examples \ref{examples}, we have that ${\mathcal T}^AA^*\to A^*$  can be identified with ${\mathcal T}^{\mathfrak g}{\mathfrak g}^*\oplus {\mathcal T}^{TM}({T^*M})\to {\mathfrak g}^*\times T^*M$, i.e. 
$$({\mathfrak g}\times T{\mathfrak g}^*)\oplus T(T^*M)\to {\mathfrak g}^*\times T^*M.$$  Under this identification,  the symplectic-like structure $\Omega_{{\mathfrak g}\times TM}$ on ${\mathcal T}^{{\mathfrak g}\times TM}({\mathfrak g}^*\times T^*M)$ is just  
$\Omega_{\mathfrak g} \oplus \Omega_{M},$
where $\Omega_M$ is the standard symplectic $2$-form on $T^*M$ and $\Omega_{{\mathfrak g}}$ is the symplectic-like structure on ${\mathfrak g}\times T{\mathfrak g}^*\to {\mathfrak g}^*$ characterized  by
$$(\Omega_{\mathfrak g})_{\eta_{0}}((\xi,\eta), (\xi',\eta'))=\eta'(\xi)-\eta(\xi') + \eta_0{[\xi,\xi']_{\mathfrak g}},$$
for all $\eta_0,\eta,\eta'\in {\mathfrak g}^*$ and $\xi,\xi'\in {\mathfrak g}^*.$}
\end{examples}

\bigskip 

\emph{The action of a Lie group $G$ on  ${\mathcal T}^AA^*\to A^*$ by complete lifts.}
Now, suppose that $\Phi:G\times A\to A$ is a  free  and proper  action of a connected  Lie group $G$ by complete lifts with respect to the Lie algebra anti-morphism $\psi:{\mathfrak g}\to
\Gamma(A).$ Denote by $\phi:G\times M\to M$ the corresponding action on $M$ and by $\Phi^*:G\times A^*\to A^*$ the left dual action on $A^*$. In what  follows,  we will describe  a free and proper canonical action  by complete lifts  on  ${\mathcal T}^AA^*$ induced by $\Phi.$ 

\begin{proposition}
Let  $\Phi:G\times A\to A$ be a free and proper action  of a connected Lie group $G$  on the Lie algebroid $A$ by complete lifts with respect to $\psi:{\mathfrak g}\to \Gamma(A).$  Then the map
$(\Phi,T\Phi^*):G\times {\mathcal T}^AA^*\to {\mathcal T}^AA^*$ given by
\begin{equation}\label{ptp}
(\Phi,T\Phi^*)(g,(a_x,v_{\alpha_x}))=(\Phi_g(a_x),(T_{\alpha_x}\Phi_{g}^*)(v_{\alpha_x})),\,\;\;  \alpha_x\in A_x^* \mbox{ and } (a_x,v_{\alpha_x})\in {\mathcal T}_{\alpha_x}^AA^*
\end{equation}
defines  a free and proper left canonical action 
of $G$ on the symplectic-like Lie algebroid ${\mathcal T}^AA^*$ by complete lifts with respect to the Lie algebra anti-morphism $\psi^T:{\mathfrak g}\to \Gamma({\mathcal T}^AA^*)$ defined by 
\begin{equation}\label{pt}\psi^T(\xi)=(\psi(\xi), \xi_{A^*}),\;\;\;\mbox{ for }\xi\in {\mathfrak g},\end{equation}
where $\xi_{A^*}$ is the infinitesimal generator of $\xi$ with respect to $\Phi^*.$ \end{proposition}

\begin{proof}
Note that the map $(\Phi,T\Phi^*)$ is well-defined. In fact, since $\Phi$ is an action by complete lifts, then $\Phi_g:A\to A$ is a Lie algebroid isomorphism, for all $g\in G.$  So, using (\ref{dAF}) with $k=0$, we have that
\begin{equation}\label{eq1}
\rho(\Phi_g(a_x))=T_x\phi_g(\rho(a_x)), \mbox{ for all $g\in G,$ $x\in M$ and $a_x\in A_x.$}
\end{equation}
Moreover,
\begin{equation}\label{eq2}
T_{\Phi^*_{g}(\alpha_x)}\tau_{*}(T_{\alpha_x}\Phi^*_{g}(v_{\alpha_x}))=T_{\alpha_x}(\phi_g\circ \tau_{*})(v_{\alpha_x})=T_x\phi_g(\rho(a_x)),\end{equation}
for all $v_{\alpha_x}\in T_{\alpha_x}A^*.$
Thus, from (\ref{eq1}) and (\ref{eq2}), we deduce that 
$$
T_{\Phi^*_{g}(\alpha_x)}\tau_{*}(T_{\alpha_x}\Phi^*_{g}(v_{\alpha_x}))=\rho(\Phi_g(a_x))$$
for all $(a_x,v_{\alpha_x})\in {\mathcal T}_{\alpha_x}^AA^*$, that is, $$(\Phi,T^*\Phi)(g,(a_x,v_{\alpha_x}))\in {\mathcal T}^A_{\Phi^*_g(\alpha_x)}A^*.$$

Obviously,  $(\Phi,T\Phi^*)$ is a free and proper action. We will now show that this action on ${\mathcal T}^AA^*$ is by complete lifts.  Firstly, note that the map 
${\psi}^T:{\mathfrak g}\to \Gamma({\mathcal T}^AA^*)$ is well defined.  In fact, since $\rho(\psi(\xi))$ is just  the infinitesimal generator $\xi_M$ of $\xi$ with respect to the action $\phi:G\times M\to M$  and  the projection  $\tau_*:A^*\to M$ is equivariant, we have that
$$\rho(\psi(\xi))=T\tau_*(\xi_{A^*}),\mbox{ for all }\xi\in {\mathfrak g}.$$

On the other hand, the infinitesimal generator $\xi_{{\mathcal T}^AA^*}\in {\mathfrak X}({{\mathcal T}^AA^*})$ of $\xi\in {\mathfrak g}$ with respect to the action $(\Phi,T\Phi^*)$ is the pair $(\xi_{A}, \xi_{A^*}^c),$ where $\xi_A$ is the infinitesimal generator of $\xi\in {\mathfrak g}$ with respect to $\Phi$  and $\xi^c_{A^*}$ is the complete lift of $\xi_{A^*}.$ Moreover, the complete lift of ${\psi}^T(\xi)$ with respect to the Lie algebroid ${\mathcal T}^AA^*$  is just $(\psi(\xi)^c,\xi_{A^*}^c).$  This is a consequence of the fact that $(\psi(\xi)^c,\xi_{A^*}^c)\in {\frak X}({\mathcal T}^AA^*)$  is $\tau_{{\mathcal T}^AA^*}-$projectable on $\rho_{{\mathcal T}^AA^*} (\psi(\xi), \xi_{A^*})=\xi_{A^*}$ and that, from (\ref{dTA}), we deduce that 
$${\mathcal L}^{{\mathcal T}^AA^*}_{(\psi(\xi),\xi_{A^*})}(\alpha,\beta)=({\mathcal L}^{A}_{\psi(\xi)}\alpha,{\mathcal L}_{\xi_{A^*}}\beta)$$
for every  projectable section $(\alpha,\beta)$ on  $\Gamma(({\mathcal T}^AA^*)^*)$.  Here ${\mathcal L}$ is the standard Lie derivative.

Therefore, $(\Phi,T\Phi^*)$ is an action by complete lifts and consequently by automorphisms of Lie algebroids. Finally, a direct computation, using (\ref{Liouville}),  proves that the action $(\Phi,T\Phi^*)$ preserves the  Liouville section $\lambda_A,$ i.e
$$ (\Phi,T\Phi^*)_g^*\lambda_A=\lambda_A, \mbox{ for all $g\in G.$}$$
Thus, using (\ref{dAF}) and the fact that  $(\Phi,T\Phi^*)_g$ is an automorphism of Lie algebroids,  we conclude that $(\Phi,T\Phi^*)_g$ preserves the canonical symplectic-like section $\Omega_A$ of ${\mathcal T}^AA^*.$
\end{proof}

\emph{The momentum map for the canonical action of $G$ on the Lie algebroid   ${\mathcal T}^AA^*\to A^*$.} Denote by  $J_{A^*}:A^*\to {\mathfrak g}^*$ the map given by
\begin{equation}\label{A*}
J_{A^*}(\alpha_x)(\xi)=\alpha_x(\psi(\xi)(x)),\;\;\;\mbox{with } x\in M,\; \alpha_x\in A_x^* \mbox{ and }\xi\in {\mathfrak g}.
\end{equation}
Then, we have the following result
\begin{proposition}\label{JA*}
The map $J_{A^*}:A^*\to {\mathfrak g}^*$ is an equivariant  momentum map for the  Poisson action $\Phi^*:G\times A^*\to A^*.$
\end{proposition}
\begin{proof}
From Proposition \ref{AP} we have that if  $\Pi_{A^*}$ is the linear Poisson structure on $A^*$ and $\widehat{\psi(\xi)}$ is the linear function associated with the section $\psi(\xi)\in \Gamma(A)$, for each $\xi \in {\mathfrak g}, $ the Hamiltonian vector field $$H_{\widehat{\psi(\xi)}}^{\Pi_{A^*}}=i_{d\widehat{\psi(\xi)}}\Pi_{A^*}\in {\mathfrak X}(A^*)$$ is just the infinitesimal generator $\xi_{A^*}\in {\mathfrak X}(A^*)$  of $\xi$ with respect to the action $\Phi^*$. Note that the function $(J_{A^*})_\xi:A^*\to {\Bbb R}$  given by 
$$(J_{A^*})_\xi(\alpha_x)=(J_{A^*}(\alpha_x))(\xi)$$
is just $\widehat{{\psi}(\xi)}.$ Thus, $J_{A^*}$ is a momentum map for the Poisson action $\Phi^*:G\times A^*\to A^*.$

Now, we will prove that $J_{A^*}$ is equivariant, i.e.,
$$J_{A^*}\circ \Phi_{g}^*=Coad_{g}^G\circ J_{A^*}.$$
 Indeed, if $x\in M$ and $\alpha_x\in A^*$, then, from  Proposition \ref{prop3.2}, we have that 
$$
\begin{array}{rcl}
J_{A^*}(\Phi_{g}^*(\alpha_x))(\xi)&=&(\Phi^*_{g}(\alpha_x))(\psi(\xi)(\phi_g(x)))=\alpha_x(\Phi_{g^{-1}}(\psi(\xi)(\phi_g(x)))\\[5pt]&=&\alpha_x(\psi(Ad^G_{g^{-1}}\xi)(x))=J_{A^*}(\alpha_x)(Ad^G_{g^{-1}}\xi)\\&=&((Coad^G_{g})(J_{A^*}(\alpha_x))(\xi),
\end{array}
$$
for all $\xi\in {\mathfrak g}.$
\end{proof}

Now, using Lemma \ref{induced-Poisson-action}, we have  that the map
\begin{equation}\label{jT}
J_{A^*}^T:{\mathcal T}^AA^*\to {\mathfrak g}^*\times {\mathfrak g}^*,\;\,\; J_{A^*}^T(a_x,v_{\alpha_x})= (T_{\alpha_x}J_{A^*}(v_{\alpha_x}), J_{A^*}(\alpha_x))
\end{equation}
is equivariant with respect to the action $(\Phi,T\Phi^*)^T:TG\times {\mathcal T}^AA^*\to {\mathcal T}^AA^*.$

From the injectivity of $\psi_x$ (see Remark \ref{phi}),  it follows that the restriction of $J_{A^*}:A^*\to {\mathfrak g}^*$ to $A_x^*$ is a linear epimorphism and therefore, for all $\alpha_x\in A^*_x$ the restriction of the  tangent map $T_{\alpha_x}J_{A^*}:T_{\alpha_x}A^*\to T_{J_{A^*}(\alpha_x)}{\mathfrak g}^*\cong {\mathfrak g}^*$ to $T_{\alpha_x}A_x^*$ is surjective. Thus, all the elements of ${\mathfrak g}^*$ are regular values of $J_{A^*}$ and  $$T_{\alpha_x}J_{A^*}\circ (\rho_{{\mathcal T}^AA^*)_{\alpha_x}}:{{\mathcal T}_{\alpha_x}^AA^*}\to {\mathfrak g}^*$$
is surjective, for all $\alpha_x\in A^*_x.$ Note that $$T_{\alpha_x}A^*_x=\ker T_{\alpha_x}\tau_{*}\subseteq (\rho_{{\mathcal T}^AA^*})_{\alpha_x}({\mathcal T}_{\alpha_x}^AA^*).$$

In conclusion if  $\mu\in {\mathfrak g}^*, $ then $J_{A^*}^{-1}(\mu)$ is a regular submanifold of $A^*$ and $(J^T_{A^*})^{-1}(0,\mu)$ is a Lie subalgebroid of ${\mathcal T}^{A}A^*$ over $J_{A^*}^{-1}(\mu)$ (see Proposition \ref{prereduction}).  In fact, $(J^T_{A^*})^{-1}(0,\mu)$  is just the prolongation $${\mathcal T}^A(J_{A^*}^{-1}(\mu))$$ of the Lie algebroid $A$ over the restriction  $(\tau_*)_{|(J_{A^*})^{-1}(\mu)}:J_{A^*}^{-1}(\mu)\to M$ of $\tau_*:A^*\to M$ to the submanifold $J_{A^*}^{-1}(\mu)$. Note that $J_{A^*}^{-1}(\mu)$ is an affine subbundle of $A^*$ over $M$ and that $(\tau_*)_{|J_{A^*}^{-1}(\mu)}:J_{A^*}^{-1}(\mu)\to M$ is the projection.

\section{The reduction of the canonical cover of a fiberwise 
linear Poisson structure}
Let $(A,\lcf\cdot,\cdot\rcf, \rho)$ be a Lie algebroid over the manifold $M$  and  $\tau:A\to M$ the vector bundle projection. Suppose that $\Phi:G\times A\to A$ is a  free and proper action of a connected Lie group $G$ by complete lifts with respect to the Lie algebra anti-morphism $\psi: {\mathfrak g}\to A.$  In the previous sections, we have shown that in this situation, we have a free and proper canonical action $(\Phi, T\Phi^*):G\times {\mathcal T}^AA^*\to {\mathcal T}^AA^*$  of the Lie group $G$ on the symplectic-like Lie algebroid ${\mathcal T}^AA^*$ by complete lifts with respect to the Lie algebra anti-morphism $\psi^T:{\mathfrak g}\to \Gamma({\mathcal T}^AA^*)$ given in (\ref{pt}). In addition, we have  an equivariant momentum map $J_{A^*}:A^*\to {\mathfrak g}^*$ on $A^*$ with respect to the left Poisson action $\Phi^*:G\times A^*\to A^*$ of $G$ on $A^*.$ 

If $\mu$ is an element of ${\mathfrak g}^*$ then  we obtain,  in a natural,  way a free and proper action $(\Phi, T\Phi^*):G_\mu\times {\mathcal T}^AJ_{A^*}^{-1}(\mu)\to {\mathcal T}^AJ_{A^*}^{-1}(\mu)$ of the isotropy group of $\mu$ on the Lie algebroid ${\mathcal T}^AJ_{A^*}^{-1}(\mu)$ by restriction. Now, using Theorem \ref{2.8co}, we conclude that the reduced vector bundle 
$$({\mathcal T}^AA^*)_\mu={\mathcal T}^AJ_{A^*}^{-1}(\mu)/TG_\mu\to J_{A^*}^{-1}(\mu)/G_\mu$$
is a symplectic-like Lie algebroid with symplectic-like section $\Omega_\mu$ characterized by 
$$\widetilde{\pi}_\mu^*\Omega_\mu=\widetilde{\iota}_\mu^*\Omega_A$$
where $\widetilde{\pi}_\mu: {\mathcal T}^AJ_{A^*}^{-1}(\mu)\to ({\mathcal T}^AA^*)_\mu$ is the canonical projection, $\widetilde{\iota}_\mu:{\mathcal T}^AJ_{A^*}^{-1}(\mu)\to {\mathcal T}^AA^*$ is the  inclusion and $\Omega_A$ is the standard symplectic-like structure on ${\mathcal T}^AA^*.$

In what follows, we will describe this reduced Lie algebroid $({\mathcal T}^AA^*)_\mu$.  Firstly, we will discuss the case $\mu=0.$ 

\subsection{The case  $\mathbf{\mu=0}$}\label{prol1} Note that, under this assumption,  the isotropy group $G_\mu$ is just  $G.$

  We will prove that the reduced symplectic-like Lie algebroid $({\mathcal T}^A{A^*})_0={\mathcal T}^AJ_{A^*}^{-1}(0)/TG\to J_{A^*}^{-1}(0)/G$ is  the canonical cover of a fiberwise linear Poisson structure on the dual $A_0^*$ of a certain Lie algebroid  $A_0$ over  $M/G.$

\medskip

{\it Description of the Lie algebroid $A_0$.} The Lie algebroid $A_0$ over $M/G$ is the space of orbits  $A/TG$ of the affine action of $TG$ on $A$ (see Theorem  \ref{reduction}).  As we know (see the proof of Theorem \ref{reduction}), if 
$\widetilde{\pi}: A\to A_0=A/TG$ and $\pi:M\to M/G$
are the canonical projections and $(\lcf\cdot,\cdot\rcf_{A_0},\rho_{A_0})$ is the Lie algebroid structure on $A_0$ then  
\begin{equation}\label{A00}
\lcf X_0,Y_0\rcf_{A_0}\circ \pi=\widetilde{\pi}(\lcf X,Y\rcf),\;\;\; \rho_{A_0}(X_0)=T\pi(\rho(X))
\end{equation}
for $X_0,Y_0\in \Gamma(A_0)$ and $X,Y\in \Gamma(A)$ satisfying 
$$X_0\circ \pi=\widetilde{\pi}\circ X,\;\;\; Y_0\circ \pi=\widetilde{\pi}\circ Y.$$

Note that with this structure,  $\widetilde{\pi}:A\to A_0$ is an epimorphism of Lie algebroids.

Now, we will prove that the vector bundle $A_0=A/TG\to M/G$ is  isomorphic to $$(J_{A^*}^{-1}(0)/G)^*\to M/G.$$ In fact, one may  easily test that  the submanifold $J_{A^*}^{-1}(0)$ is just the annihilator $(\psi({\mathfrak g}))^0$ of $\psi({\mathfrak g})$. Therefore, the restriction $\tau_{*}^0=\tau_{*|J_{A^*}^{-1}(0)}:J_{A^*}^{-1}(0)\to M$ of $\tau_*:A^*\to M$ to
this submanifold is a vector bundle over $M$. Moreover, a direct computation proves that this vector bundle  is isomorphic to the dual vector bundle $(A/\psi({\mathfrak g}))^*$ of $A/\psi({\mathfrak g})\to M.$

Therefore,  using the equivalences (\ref{tresformas}),  we deduce that the three vector bundles $$A_0=A/TG\to M/G, 
\;\;\;\;\;(A/\psi({\mathfrak g}))/G\to M/G,\;\;\;\;\; (J_{A^*}^{-1}(0)/G)^*\cong (J_{A^*}^{-1}(0))^*/G\to M/G$$ are isomorphic. Thus, we may induce  isomorphic Lie algebroid structures on these vector bundles. 

\medskip 

{\it The description of the Lie algebroid isomorphism between $({\mathcal T}^{A}A^*)_0$ and ${\mathcal T}^{A_0}{A_0^*}$.}   In what follows we identify $A_0=A/TG$ with $(J_{A^*}^{-1}(0)/G)^*\cong (J_{A^*}^{-1}(0))^*/G.$  Under this identification, we denote by $\varphi:A\to (J_{A^*}^{-1}(0)/G)^*$ the epimorphism of vector bundles corresponding to the quotient projection $A\to A/TG.$

Let us consider
the following epimorphism of vector bundles over $\pi_0:J_{A^*}^{-1}(0)\to A_0^*=J_{A^*}^{-1}(0)/G$
\begin{equation}\label{varphiT}
\varphi^T:{\mathcal T}^{A}J_{A^*}^{-1}(0)\to {\mathcal T}^{A_0}A_0^*, \;\;\; \ (a_x,v_{\alpha_x})\mapsto (\varphi(a_x), T_{\alpha_x}\pi_0(v_{\alpha_x})).
\end{equation}
Note that, using (\ref{A00}) and the fact that $\widetilde\tau_*^0\circ \pi_0=\pi\circ \tau_*^0, $ we have that 
 $\varphi^T$ is well-defined. Here $\widetilde\tau_*^0:J_{A^*}^{-1}(0)/G\to M/G$ is the dual vector bundle of 
 $(J_{A^*}^{-1}(0)/G)^*\to M/G$. 

Now, since  that $\varphi:A\to A_0$ is a Lie algebroid epimorphism, then  $\varphi^T$ is also a Lie algebroid epimorphism. Furthermore, it is easy to prove, using that $\varphi$ is $TG$-invariant,  that  $\varphi^T$  is $TG$-invariant with respect to the action $(\Phi,T\Phi^*)^T$ of $TG$ restricted to ${\mathcal T}^AJ_{A^*}^{-1}(0).$ In fact, 
$$
\begin{array}{rcl}
\varphi^T((\Phi, T\Phi^*)^T_{(g,\xi)}(a_x,v_{\alpha_x}))&=& (\varphi(\Phi_g(a_x)) + \varphi(\Phi_g(\psi(\xi))), T_{\alpha_x}\pi_0(T_x\Phi_g(v_{\alpha_x}+\xi_{A^*})))\\&=& (\varphi(a_x), T_{\alpha_x}\pi_0(v_{\alpha_x}))=\varphi^T(a_x,v_{\alpha_x})
\end{array}
$$
for all $(g,\xi)\in G\times {\mathfrak g}\cong TG$ and $(a_x,v_{\alpha_x})\in {\mathcal T}^A_{\alpha_x}J_{A^*}^{-1}(0).$ Note that   $$\varphi(\psi(\xi))=\varphi(\Phi_{(e,\xi)}(0))=0.$$

Thus,  we have the following Lie algebroid epimorphism $\bar{\varphi}^T$ between $({\mathcal T}^AA^*)_0$ and ${\mathcal T}^{A_0}A_0^*$ over the identity of $A_0^*$ 

\begin{equation}\label{barvarphiT}
\begin{picture}(375,60)(60,40)
\put(200,20){\makebox(0,0){$A_0^*$}} \put(260,25){$Id$}
\put(215,20){\vector(1,0){90}} \put(320,20){\makebox(0,0){$A^*_0$}}
\put(185,50){$\tau_0^T$} \put(200,75){\vector(0,-1){45}}
\put(320,50){$\tau_{{\mathcal T}^{A_0}A_0^*}$} \put(310,75){\vector(0,-1){45}}
\put(195,80){\makebox(0,0){$({\mathcal T}^AA^*)_0$}} \put(260,85){$\bar\varphi^T$}
\put(215,80){\vector(1,0){90}} \put(320,80){\makebox(0,0){${\mathcal T}^{A_0}A_0^*$}}
\end{picture}
\end{equation}

\vspace{30pt}
 
 In fact, 
 \begin{equation}\label{5.33'}
 \bar{\varphi}^T[(a_x,v_{\alpha_x})]=(\varphi(a_x), (T_{\alpha_x}\pi_0)(v_{\alpha_x})), \mbox{ for } (a_x,v_{\alpha_x})\in {\mathcal T}^A_{\alpha_{x}} J_{A^*}^{-1}(0).
 \end{equation}

Finally, we will prove that $\bar\varphi^T$ is an isomorphism, that is, $\bar\varphi^T$ is injective.

If $(e_x,v_{\alpha_x})\in \ker \varphi^T$ then $e_x\in \ker\varphi=\psi_x({\mathfrak g})$ and $v_{\alpha_x}$ is a vertical vector with respect to $\pi_0:(J_{A^*})^{-1}(0)\to (J_{A^*})^{-1}(0)/G.$ Then, there are $\xi,\xi'\in {\mathfrak g}$ such that
$$e_x=\psi_x(\xi) \mbox{ and } v_{\alpha_x}=\xi'_{A^*}(\alpha_x).$$
Then, 
$$\xi_M(x)=\rho(e_x)=T_{\alpha_x}\tau^0_*(v_{\alpha_x})=\xi'_M(x)$$
and, since $\phi:G\times M\to M$ is a free action, we conclude that $\xi=\xi'.$ Therefore,
$$(e_x,v_{\alpha_x})=(\psi_x(\xi), \xi_{A^*}(\alpha_x))=(\Phi,T\Phi^*)^T((e,\xi), (0,0))$$
where $e$ is the identity element of $G$. Thus, $\bar\varphi^T$ is injective.

\medskip

{\it The Lie algebroid isomorphism between $({\mathcal T}^{A}{A^*})_0$ and ${\mathcal T}^{A_0}{A_0^*}$ is canonical.}
We will see that $\bar\varphi^T$ is canonical, i.e.
\begin{equation}\label{symplectic}
(\bar\varphi^T)^*\Omega_{A_0}=\Omega_0,
\end{equation}
where  $\Omega_0$ is the reduced symplectic-like structure on $({\mathcal T}^{A}{A^*})_0$ given in Theorem \ref{2.8co} and $\Omega_{A_0}$ is the canonical symplectic-like structure on ${\mathcal T}^{A_0}A_0^*$. 

From (\ref{Liouville}), (\ref{varphiT}) and the definition of the morphism $\varphi:A\to A_0$, we obtain that 
$$(\varphi^T)^*\lambda_{A_0}=\widetilde{\iota}_0^*\lambda_A$$
where $\lambda_{A_0}$ (respectively, $\lambda_{A}$) is the Liouville section of $A_0\to M/G$ (respectively, of $A\to M$) and $\widetilde{\iota}_0:{\mathcal T}^AJ_{A^*}^{-1}(0)\to {\mathcal T}^AA^*$ is the inclusion. 

Thus, since $\varphi^T$ and $\widetilde{\iota}_0$ are Lie algebroid morphisms, we obtain that 
$$(\varphi^T)^*\Omega_{A_0}=\widetilde{\iota}_0^*\Omega_A.$$
On the other hand, if $\widetilde{\pi}_0: {\mathcal T}^AJ_{A^*}^{-1}(0)\to ({\mathcal T}^AA^*)_0$ is the canonical projection, it is clear that $\bar\varphi^T\circ \widetilde{\pi}_0=\varphi^T$ which implies that 
$$\widetilde{\pi}_0^*((\bar\varphi^T)^*\Omega_{A_0})=\widetilde{\pi}_0^*\Omega_0$$
and, therefore, 
\begin{equation}\label{omegaT}
(\bar{\varphi}^T)^*\Omega_{A_0}=\Omega_0.
\end{equation} 

In the following theorem we summarize the results obtained in the case $\mu=0.$

\begin{theorem}\label{p1}
Let $(A,\lcf\cdot,\cdot\rcf,\rho)$ be a Lie algebroid on the manifold $M$ and $\Phi:G\times A\to A$  a free and proper action of a connected Lie group by complete lifts. Then, the reduced symplectic-like Lie algebroid $$({\mathcal T}^AA^*)_0=({\mathcal T}^AJ_{A^*}^{-1}(0))/TG\to J_{A^*}^{-1}(0)/G$$
is canonically isomorphic to the Lie algebroid ${\mathcal T}^{A_0}A_0^*$, equipped with the standard symplectic-like structure, where the Lie algebroid  $A_0$ is the vector bundle  
$$A_0=A/TG\to M/G$$
endowed with the quotient Lie algebroid structure characterized by (\ref{A00}). 
\end{theorem}

\subsection{The case 
$\mathbf{G_\mu=G}$.}\label{prol2}
Suppose that the assumptions of Theorem \ref{p1} hold. Additionally,  we consider a principal $G$-connection ${\mathcal A}: TM\to {\mathfrak g}$ for the corresponding principal bundle $\pi:M\to M/G.$ In such a case we have a vector bundle morphism ${\mathcal A}^A:A\to {\mathfrak g}$ given by 
$${\mathcal A}^A(a_x)={\mathcal A}(\rho_x(a_x)),\;\,\;\forall a_x\in A_x,$$
which satisfies the following properties:
\begin{enumerate}
\item ${\mathcal A}^A$ is equivariant with respect to $\Phi:G\times A\to A$ and the adjoint action, that is, 
$${\mathcal A}^A(\Phi_g(a_x))=Ad_g^G({\mathcal A}^A(a_x)),\;\;\; \forall a_x\in A_x,$$
\item ${\mathcal A}^A(\psi(\xi)(x))=\xi,$   for all $\xi\in {\mathfrak g}$ and $x\in M.$
\end{enumerate}
 
Note that if $\pi:M\to M/G$ is the quotient projection then we have 
$$TQ=V\pi\oplus H\mbox{ and } A=\psi({\mathfrak g})\oplus H^A,$$
where $V\pi$ is the vertical bundle of $\pi$ and $H^A$ (respectively, $H$) is the vector bundle on $M$ whose fiber at $x\in M$ is the vector space 
$$H^A_x=\{a_x\in A_x/{\mathcal A}^A(a_x)=0\} \mbox{ (respectively, $H_x=\{v_x\in T_xM/{\mathcal A}(v_x)=0\}$}).$$
Moreover, $H$ and $H^A$ are $G$-invariant vector bundles, that is, 
$$H^A_{\phi_g(x)}=\Phi_g(H_x^A)\mbox{ and } H_{\phi_g(x)}=T_x\phi_g(H_x), \;\;\; \forall g\in G.$$

Now, if $\mu\in {\mathfrak g}^*$, we consider the section $\alpha_\mu$ of $A^*$ given by 
$$\alpha_\mu(a_x)=\mu({\mathcal A}^A(a_x)),$$
with $x\in M$ and $a_x\in A_x$.
This section has the following properties: 
\begin{enumerate}
\item $\alpha_\mu(M)\subseteq J_{A^*}^{-1}(\mu)$. In fact, 
$$J_{A^*}(\alpha_\mu(x))(\xi)=\alpha_\mu(\psi(\xi)(x))=\mu({\mathcal A}^A(\psi(\xi)(x)))=\mu(\xi)$$
for all $x\in M$ and $\xi\in {\mathfrak g}.$ 
\item $\Phi_g^*\alpha_\mu=\alpha_{Coad^G_g\mu},$ for all $g\in G$, which is a consequence from the equivariance properties  of ${\mathcal A}^A.$ 
\end{enumerate}
Thus, since $G_\mu=G$ then we  deduce that $\alpha_\mu$ is $G$-invariant, i.e.
\begin{equation}\label{invarianza}
\Phi_g^*(\alpha_\mu)=\alpha_\mu.
\end{equation} 
So, in what follows we assume that there is a $G$-invariant $1$-section $\alpha_\mu\in \Gamma(A^*)$ of $A^*$ with values in $J_{A^*}^{-1}(\mu).$ 
Using (\ref{invarianza}), Proposition \ref{Poisson} and the fact that the flow of $\psi(\xi)$ is $\{\Phi_{exp(t\xi)}\}_{t\in {\Bbb R}}$, we obtain that  
\begin{equation}\label{LPSI}{\mathcal L}^A_{\psi(\xi)}\alpha_\mu=0.\end{equation}
On the other hand, 
$$i_{\psi(\xi)}\alpha_\mu=\mu({\mathcal A}^A(\psi(\xi))=\mu(\xi).$$

Then 
\begin{equation}\label{ipsi}
0={\mathcal L}_{\psi(\xi)}^A\alpha_\mu=i_{\psi(\xi)}d^A\alpha_\mu.
\end{equation}
Denote by $\beta_\mu=d^A\alpha_\mu$. From  (\ref{invarianza}) and since $\Phi_g:A\to A$ is a Lie algebroid morphism we deduce that  the $2$-section $\beta_\mu$ of $A^*$ is $G$-invariant. Moreover, it satisfies $i_{\psi(\xi)}\beta_\mu=0$ which implies that 
$$(\Phi^T_{(g,\xi)})^*\beta_\mu=\Phi_g^*\beta_\mu=\beta_\mu$$
for all $(g,\xi)\in G\times {\mathfrak g}\cong TG$.

Therefore,  there exists a unique $B_\mu\in \Gamma(\wedge^2A_0^*)$ with the following property: 
\begin{equation}\label{property}
\widetilde{\pi}^*B_\mu=\beta_\mu=d^A\alpha_\mu, 
\end{equation}
where $\widetilde{\pi}:A\to A_0$ is the corresponding projection. It is clear that  $d^{A_0}B_\mu=0$.

The $2$-section $B_\mu$ of $A_0^*$ is said to be the {\it magnetic term associated with }  $\alpha_\mu.$ 

Now, we will prove that there is a Lie algebroid isomorphism, $\Upsilon_{\alpha_\mu}:({\mathcal T}^AA^*)_\mu\to {\mathcal T}^{A_0}A_0^*$  between the reduced Lie algebroid $({\mathcal T}^AA^*)_\mu$ and ${\mathcal T}^{A_0}A_0^*$ such that   the symplectic-like section $\Omega_{A_0}$ on ${\mathcal T}^{A_0}A_0^*$ and the reduced symplectic-like section $\Omega_{\mu}$  on $({\mathcal T}^AA^*)_\mu$ are related by the following formula
$$\Upsilon_{\alpha_\mu}^*(\Omega_{A_0} - pr_1^*B_\mu)=\Omega_\mu,$$ 
where $pr_1:{\mathcal T}^{A_0}A_0^*\to A_0$ is the Lie algebroid morphism induced by the first projection. 

\medskip

{\it The description of  the Lie algebroid isomorphism $\Upsilon_{\alpha_\mu}:({\mathcal T}^AA^*)_\mu\to {\mathcal T}^{A_0}A_0^*.$ } Firstly,  we will describe a Lie algebroid morphism between the reduced spaces $({\mathcal T}^AA^*)_\mu$ and $({\mathcal T}^AA^*)_0$. Then, we may use Theorem \ref{p1} in order to construct the isomorphism $\Upsilon_{\alpha_\mu}$.

Using the fact that $\alpha_\mu(M)\subseteq J_{A^*}^{-1}(\mu),$  we deduce that  $J_{A^*}^{-1}(\mu)\to M$ is an affine bundle on $M$ such that
$$J_{A^*}^{-1}(\mu)\cap A_x^*=\{\beta_x\in A^*_x/\beta_x-\alpha_\mu(x)\in J_{A^*}^{-1}(0)\}$$
for all $x\in M.$

Now, we consider the affine bundle isomorphism

\begin{picture}(375,60)(60,40)
\put(200,20){\makebox(0,0){$M$}} \put(260,25){$Id$}
\put(215,20){\vector(1,0){90}} \put(320,20){\makebox(0,0){$M$}}
\put(185,50){$$} \put(200,70){\vector(0,-1){45}}
\put(320,50){$$} \put(320,70){\vector(0,-1){45}}
\put(190,80){\makebox(0,0){$J_{A^*}^{-1}(\mu)$}} \put(260,85){$sh_\mu$}
\put(215,80){\vector(1,0){90}} \put(325,80){\makebox(0,0){$J_{A^*}^{-1}(0)$}}
\end{picture}

\vspace{30pt}

where $sh_\mu(\beta_x)=\beta_x-\alpha_\mu(x),$ for all $\beta_x\in J^{-1}_{A^*}(\mu)\cap A_x^*.$ 

From the $G$-invariance of $\alpha_\mu$, we deduce that $sh_\mu$ is equivariant with respect to the action $\Phi^*:G\times A^*\to A^*$, i.e. 
$$(sh_\mu\circ \Phi_g^*)({\beta_x})=(\Phi_g^*\circ {sh_\mu)(\beta_x}),$$ for all $\beta_x\in A^*_x\cap J_{A^*}^{-1}(\mu)$ and $g\in G.$
Moreover,  one may induce a morphism of vector bundles

\begin{picture}(375,60)(60,40)
\put(190,15){\makebox(0,0){$J_{A^*}^{-1}(\mu)$}} \put(260,20){$sh_\mu$}
\put(215,15){\vector(1,0){90}} \put(325,15){\makebox(0,0){$J_{A^*}^{-1}(0)$}}
\put(155,50){$\tau_{{\mathcal T}^AJ_{A^*}^{-1}(\mu)}$} \put(200,70){\vector(0,-1){45}}
\put(325,50){$\tau_{{\mathcal T}^AJ_{A^*}^{-1}(0)}$} \put(320,70){\vector(0,-1){45}}
\put(190,80){\makebox(0,0){${\mathcal T}^AJ_{A^*}^{-1}(\mu)$}} \put(250,85){${\mathcal T}^Ash_\mu$}
\put(215,80){\vector(1,0){90}} \put(330,80){\makebox(0,0){${\mathcal T}^AJ_{A^*}^{-1}(0)$}}
\end{picture}

\vspace{40pt}

where ${\mathcal T}^Ash_\mu(a_x,X_{\beta_x})=(a_x,T_{\beta_x}sh_\mu(X_{\beta_x})),$  with $(a_x,X_{\beta_x})\in {{\mathcal T}^AJ_{A^*}^{-1}(\mu)}.$
Note that, since $$\tau_{*|J_{A^*}^{-1}(0)}\circ sh_\mu=\tau_{*|J_{A^*}^{-1}(\mu)},$$ then ${\mathcal T}^Ash_\mu$ is well-defined. Furthermore, a direct proof shows that ${\mathcal T}^Ash_\mu$ is an isomorphism of vector bundles.  In fact, one can easily see  that  ${\mathcal T}^Ash_\mu$ is a Lie algebroid isomorphism, taking into account that
$$
\begin{array}{rcl}
{\mathcal T}^Ash_\mu(\lcf X,Y\rcf,[U,V])&=&(\lcf X,Y\rcf,[Tsh_\mu\circ U\circ sh_\mu^{-1},Tsh_\mu\circ {V}\circ sh_\mu^{-1}])\circ sh_\mu,\\[8pt]
\rho_{{\mathcal T}^AJ_{A^*}^{-1}(0)}(({\mathcal T}^Ash_\mu)(X,U))&=&(T sh_\mu \circ \rho_{{\mathcal T}^AJ_{A^*}^{-1}(\mu)})(X,U)
\end{array}$$
for all $X,Y\in \Gamma(A)$ and $U,V\in {\mathfrak  X}(J_{A^*}^{-1}(\mu))$ which are $(\tau_*)_{|J_{A^*}^{-1}(\mu)}$-projectable on $\rho(X)$ and $\rho(Y),$ respectively. 

 Moreover, since $sh_\mu$ is equivariant, we  deduce that ${\mathcal T}^Ash_\mu$ is equivariant with respect to the action $(\Phi,T\Phi^*)^T$ of $G$ restricted to ${\mathcal T}^AJ_{A^*}^{-1}(\mu)$ and ${\mathcal T}^AJ_{A^*}^{-1}(0)$, respectively. 

Thus, one induces a Lie algebroid isomorphism

\begin{picture}(375,60)(60,40)
\put(190,15){\makebox(0,0){$J_{A^*}^{-1}(\mu)/G$}} \put(260,20){$\widetilde{sh}_\mu$}
\put(215,15){\vector(1,0){90}} \put(330,15){\makebox(0,0){$J_{A^*}^{-1}(0)/G$}}
\put(155,50){$\widetilde{\tau}_{{\mathcal T}^AJ_{A^*}^{-1}(\mu)}$} \put(200,70){\vector(0,-1){45}}
\put(325,50){$\widetilde{\tau}_{{\mathcal T}^AJ_{A^*}^{-1}(0)}$} \put(320,70){\vector(0,-1){45}}
\put(180,80){\makebox(0,0){$({\mathcal T}^AJ_{A^*}^{-1}(\mu))/TG$}} \put(250,85){$\ltilde{30}{{\mathcal T}^Ash_\mu}$}
\put(220,80){\vector(1,0){80}} \put(340,80){\makebox(0,0){$({\mathcal T}^AJ_{A^*}^{-1}(0))/TG$}}
\end{picture}

\vspace{30pt}

Finally, the isomorphism $\Upsilon_{\alpha_\mu}:({\mathcal T}^AA^*)_\mu\to {\mathcal T}^{A_0}A_0^*$ is defined as follows  
$$\Upsilon_{\alpha_\mu}=\bar\varphi^T\circ \ltilde{30}{{\mathcal T}^Ash_\mu}$$
where $\bar\varphi^T:({\mathcal T}^AA^*)_0\to {\mathcal T}^{A_0}A_0^*$ is the Lie algebroid isomorphism defined by (\ref{5.33'}). 

\medskip

{\it Relation between the symplectic-like structures on $({\mathcal T}^AA^*)_\mu$ and  ${\mathcal T}^{A_0}A_0^*$. } Let $\lambda_A$ be the Liouville section on ${\mathcal T}^AA^*$ and  $\iota_0:{\mathcal T}^A(J_{A^*})^{-1}(0)\to {\mathcal T}^AA^*$ (respectively, $\iota_\mu:{\mathcal T}^A(J_{A^*})^{-1}(\mu)\to {\mathcal T}^AA^*$) be the corresponding  inclusion. Then,
\begin{equation*}\label{parte1}
\begin{array}{rcl}
(({\mathcal T}^Ash_\mu)^*(\iota_0^*\lambda_A))(a_x,X_{\beta_x}) &=&(\iota_\mu^*\lambda_A)(a_x,X_{\beta_x})-\alpha_\mu(a_x),
\end{array}\end{equation*}

for all $\beta_x\in J_{A^*}^{-1}(\mu)$ and $(a_x,X_{\beta_x})\in {\mathcal T}^A_{\beta_x}J_{A^*}^{-1}(\mu).$

On the other hand, if $pr^0_1:{\mathcal T}^AJ_{A^*}^{-1}(0)\to A$ is the Lie algebroid morphism induced by the first projection, we have that
\begin{equation*}\label{parte2}
((pr_1^0\circ {\mathcal T}^Ash_\mu)^*\alpha_\mu)(a_x,X_{\beta_x})=\alpha_\mu(a_x).
\end{equation*}
This implies that 
$$({\mathcal T}^Ash_\mu)^*(\iota_0^*\lambda_A + (pr_1^0)^*\alpha_\mu)=\iota_\mu^*\lambda_A$$
and thus, from Theorem \ref{2.8co}, we deduce that 
\begin{equation}\label{5.41}
({\mathcal T}^Ash_\mu)^*(\widetilde{\pi}_0^*\Omega_0-(pr_1^0)^*\beta_\mu)=\widetilde{\pi}_\mu^*\Omega_\mu,
\end{equation}
where $\Omega_0$ (respectively, $\Omega_\mu$) is the symplectic-like structure on $({\mathcal T}^AJ_{A^*}^{-1}(0))/TG$ (respectively, \linebreak 
$({\mathcal T}^AJ_{A^*}^{-1}(\mu))/TG$) and $\widetilde{\pi}_0:{\mathcal T}^AJ_{A^*}^{-1}(0)\to ({\mathcal T}^AJ_{A^*}^{-1}(0))/TG$ (respectively, $\widetilde{\pi}_\mu:{\mathcal T}^AJ_{A^*}^{-1}(\mu)\to ({\mathcal T}^AJ_{A^*}^{-1}(\mu))/TG$) is the canonical projection.

Now, using the  relations
$$\widetilde{\pi}_0\circ {\mathcal T}^Ash_\mu=\ltilde{30}{{\mathcal T}^Ash_\mu}\circ \widetilde{\pi}_\mu\;\;\mbox{ and }\;\; \widetilde{\pi}\circ pr_1^0\circ {\mathcal T}^Ash_\mu=pr_1\circ \Upsilon_{\alpha_\mu}\circ \widetilde{\pi}_\mu,$$ 
and the facts 
$$(\bar\varphi^T)^*\Omega_{A_0}=\Omega_0\;\;\mbox{ and  } \widetilde{\pi}^*B_\mu=\beta_\mu,$$
we conclude that 
(\ref{5.41})  is equivalent to
$$\widetilde{\pi}_\mu^*(\Upsilon_{\alpha_\mu}^*\Omega_{A_0}-\Upsilon^*_{\alpha_\mu}(pr_1^*(B_\mu)))=\widetilde{\pi}_\mu^*\Omega_\mu.$$
Therefore, 
$$\Upsilon_{\alpha_\mu}^*(\Omega_{A_0}-pr_1^*(B_\mu))=\Omega_\mu.$$
The results obtained in this case may be summarized in the following theorem.

\begin{theorem}\label{p2}
Let $(A,\lcf\cdot,\cdot\rcf,\rho)$ be a Lie algebroid on the manifold $M$ and $\Phi:G\times A\to A$  a free and proper action of a connected Lie group $G$ by complete lifts. Suppose that we consider $\mu\in {\mathfrak g}^*$ such that $G=G_\mu$. Then, choosing any  $G$-invariant section $\alpha_\mu$ of $A^*$ with values in $J_{A^*}^{-1}(\mu)$, there is a canonical Lie algebroid isomorphism  
$$\Upsilon_{\alpha_\mu}:(({\mathcal T}^AA^*)_\mu,\Omega_\mu)\to ({\mathcal T}^{A_0}A_0^*, \Omega_{A_0} - (pr_1)^*B_\mu)$$
where $A_0$ is the vector bundle  
$$A_0=A/TG\to M/G$$
endowed with the Lie algebroid structure characterized by (\ref{A00}),   
$\Omega_{A_0}$ is the canonical symplectic-like structure on ${\mathcal T}^{A_0}A_0^*$, $pr_1:{\mathcal T}^{A_0}A_0^*\to A_0$ is the projection on the first factor and $B_\mu\in\Gamma(\wedge^2A^*_0)$ is the corresponding magnetic term associated with $\alpha_\mu$ which is characterized by (\ref{property}). \end{theorem}

\subsection{The general case}\label{prol3}
Let $(A,\lcf\cdot,\cdot\rcf, \rho)$ be a Lie algebroid on the manifold $M$ and $\Phi:G\times A\to A$ a free and proper  action of a connected Lie group $G$ on $A$ by complete lifts with respect to the Lie algebra anti-morphism $\psi:{\mathfrak g}\to \Gamma(A).$

Let $\mu\in {\mathfrak g}^*$ and denote by ${\mathfrak g}_\mu$ the isotropy algebra of $\mu.$ Then, the induced action  $\Phi: G_\mu\times A\to A$ is a free and proper action by complete lifts with respect to the restriction $\psi: {\mathfrak g}_\mu\to \Gamma(A)$ of $\psi$ to ${\mathfrak g}_\mu.$ 

Now, denote by 
$\bar{\mu}\in {\mathfrak g}_\mu^*$ the restriction of $\mu$ to ${\mathfrak g}_\mu$ and by  $J_{A^*}^\mu:A^*\to {\mathfrak g}_\mu^*$ the map given by 
$$J_{A^*}^\mu=i^*\circ J_{A^*},$$
where $i^*:{\mathfrak g}^*\to {\mathfrak g}_\mu^*$ is the dual of the inclusion $i:{\mathfrak g}_\mu\to {\mathfrak g}.$ Then, 
$J_{A^*}^\mu$ is  
 the momentum map associated with the action of $G_\mu$ on $A.$

A direct computation proves that the isotropy group of $\bar\mu\in {\mathfrak g}_\mu$ with respect to the coadjoint action of $G_\mu$, $(G_\mu)_{\bar\mu},$ is just $G_\mu.$ Therefore, we are in the conditions of Section \ref{prol2} if we choose as the starting Lie group $G_\mu$. Next, we choose  a $G_\mu$-invariant section $\alpha_{\mu}\in \Gamma(A^*)$ such that
$$\alpha_{\mu}(M)\subset (J_{A^*}^\mu)^{-1}(\bar\mu).$$

This is always possible as we have shown in Section \ref{prol2}. 
If $A_{0,\mu}$ is the vector bundle $A/TG_\mu\to M/G_\mu$ associated with the action $\Phi^T:TG_\mu\times A\to A,$ we denote by $B_{\mu}\in \Gamma(\wedge^2A_{0,\mu}^*)$ the corresponding magnetic term associated with $\alpha_{\mu}$. Then, from Theorem \ref{p2}, we conclude that the reduced symplectic-like Lie algebroid 
$$({\mathcal T}^AA^*)_{\bar\mu}=({\mathcal T}^A(J_{A^*}^{\mu})^{-1}(\bar\mu))/TG_\mu\to J^{-1}_{A^*}(\mu)/G_\mu$$
is isomorphic to the symplectic-like Lie algebroid $({\mathcal T}^{A_{0,\mu}}A_{0,\mu}^*, \Omega_{A_{0,\mu}}-pr_1^*(B_{{\mu}})),$
where $\Omega_{A_{0,\mu}}$ is the canonical symplectic-like structure on ${\mathcal T}^{A_{0,\mu}}A_{0,\mu}^*$ and $pr_1:{\mathcal T}^{A_{0,\mu}}A_{0,\mu}^*\to A_{0,\mu}$ is the projection on the first factor. 

On the other hand, the inclusion $i_{\mu,\bar{\mu}}:J_{A^*}^{-1}(\mu)\to (J^\mu_{A^*})^{-1}(\bar\mu)$ is $G_\mu$-invariant and induces a Lie algebroid  $TG_\mu$-invariant monomorphism $I:{\mathcal T}^AJ_{A^*}^{-1}(\mu)\to {\mathcal T}^A(J^\mu_{A^*})^{-1}(\bar\mu)$ over $i_{\mu,\bar{\mu}}$. Therefore, we have a Lie algebroid monomorphism $(\widetilde{I},\widetilde{i}_{\mu,\bar{\mu}})$ 

\begin{picture}(375,60)(60,40)
\put(190,15){\makebox(0,0){$J_{A^*}^{-1}(\mu)/G_\mu$}} \put(260,20){$\widetilde{i}_{\mu,\bar{\mu}}$}
\put(215,15){\vector(1,0){90}} \put(350,15){\makebox(0,0){$(J^\mu_{A^*})^{-1}(\bar{\mu})/G_\mu$}}
\put(155,50){$\widetilde{\tau}_{{\mathcal T}^AJ_{A^*}^{-1}(\mu)}$} \put(200,70){\vector(0,-1){45}}
\put(325,50){$\widetilde{\tau}_{{\mathcal T}^A(J_{A^*}^\mu)^{-1}(\bar{\mu})}$} \put(320,70){\vector(0,-1){45}}
\put(165,80){\makebox(0,0){$({\mathcal T}^AJ_{A^*}^{-1}(\mu))/TG_\mu$}} \put(260,85){$\widetilde{I}$}
\put(215,80){\vector(1,0){90}} \put(360,80){\makebox(0,0){$({\mathcal T}^A(J^\mu_{A^*})^{-1}(\bar\mu))/TG_\mu$}}
\end{picture}

\vspace{35pt}

which is canonical with respect to  $\Omega_\mu$  and  $\Omega_{\bar\mu}$ on the reduced spaces $({\mathcal T}^A(J_{A^*})^{-1}(\bar\mu))/TG_\mu$ and $({\mathcal T}^A(J^\mu_{A^*})^{-1}(\bar\mu))/TG_\mu,$ respectively.

Denote by $\tilde{\iota}_{\bar\mu}:{\mathcal T}^A(J_{A^*}^\mu)^{-1}(\bar\mu)\to {\mathcal T}^AA^*$ and by $\tilde{\iota}_\mu:{\mathcal T}^A(J_{A^*})^{-1}(\mu)\to {\mathcal T}^AA^*$ the corresponding inclusions which are related by
$$\tilde{\iota}_\mu=\tilde{\iota}_{\bar \mu}\circ {I}.$$

Now, if 

$$\begin{array}{rcl}\pi_{\bar\mu}: {\mathcal T}^A(J_{A^*}^\mu)^{-1}(\bar\mu)&\to& ({\mathcal T}^A(J_{A^*}^\mu)^{-1}(\bar\mu))/TG_\mu\\[8pt]
\pi_\mu: {\mathcal T}^A(J_{A^*})^{-1}(\mu)&\to &({\mathcal T}^A(J_{A^*})^{-1}(\mu))/TG_\mu\end{array}$$ are the corresponding projections, we have that
$$\pi_\mu^*\Omega_\mu=\tilde{\iota}_\mu^*\Omega_A ={I}^*(\tilde{\iota}_{\bar{\mu}}^*\Omega_A)={I}^*(\pi_{\bar{\mu}}^*\Omega_{\bar \mu}).$$
Then, using that $\pi_{\bar\mu}\circ {I}=\widetilde{I}\circ \pi_\mu$ we conclude that
$$\widetilde{I}^*\Omega_{\bar\mu}=\Omega_\mu.$$

Therefore, $\widetilde{I}$ is a canonical Lie algebroid monomorphism. Thus, we have proved the  main result  of this section.

\begin{theorem}\label{t5.3}
Let $(A,\lcf\cdot,\cdot\rcf,\rho)$ be a Lie algebroid over the manifold $M$ and $\Phi:G\times A\to A$  a free and proper action of a connected Lie group by complete lifts. If  $\mu\in {\mathfrak g}^*$  and $\bar\mu$ is the restriction of $\mu$ to ${\mathfrak g}_\mu$, then, choosing a  $G_\mu$-invariant section $\alpha_\mu$ of $A^*$ with values in $(J^\mu_{A^*})^{-1}(\bar\mu)$, there exists a canonical embedding
$$({\mathcal T}^AA^*)_\mu\to  {\mathcal T}^{A_{0,\mu}}A_{0,\mu}^*$$
 from  the reduced  algebroid $({\mathcal T}^AA^*)_\mu$ equipped with the canonical reduced symplectic-like structure $\Omega_\mu$ to  the Lie algebroid ${\mathcal T}^{A_{0,\mu}}A_{0,\mu}^*$ endowed with the symplectic-like structure $$\bar\Omega_{\mu}=\Omega_{A_{0,\mu}} - (pr_1)^*B_\mu.$$
  Moreover, this embedding is an isomorphism if and only if ${\mathfrak g}={\mathfrak g}_\mu.$
  
 Here  $A_{0,\mu}$ is the vector bundle  
$$A_{0,\mu}=A/TG_\mu\to M/G_\mu,$$
$\Omega_{A_{0,\mu}}$ is the canonical symplectic-like structure on ${\mathcal T}^{A_{0,\mu}}A_{0,\mu}^*$, $pr_1:{\mathcal T}^{A_{0,\mu}}A_{0,\mu}^*\to A_{0,\mu}$ is the projection on the first factor and $B_\mu\in \Gamma(\wedge^2A_0^*)$ is  the corresponding magnetic term associated with $\alpha_\mu$ which is characterized by (\ref{property}). \end{theorem}

\begin{examples} { {\rm $(i)$ If we apply the previous theorem to the particular case when $A$ is the standard Lie algebroid $TM\to M$ then we recover a classical result in cotangent bundle reduction theory (see \cite{AbMa87,Ku81}).  

\medskip 

$(ii)$ For  the case $A={\mathfrak g}\times TM$ from $(ii)$ in Examples \ref{examples}, we have seen  that the vector bundle ${\mathcal T}^AA^*\to A^*$  can be identified with 
${\mathfrak g}\times T({\mathfrak g}^*\times T^*M)\to {\mathfrak g}^*\times T^*M$ (see Example \ref{4.1}). Moreover, the Lie algebroid $({\mathfrak g}\times TM)/TG\to M/G$ is isomorphic to the  Atiyah algebroid associated with the principal bundle $\pi_M:M\to M/G$ (see (\ref{F})). 

\medskip

Then, the reduced symplectic-like Lie algebroid  $({\mathcal T}^AA^*)_0,$ for the value $\mu=0\in {\mathfrak g}^*$, is simplectically isomorphic to the canonical cover of  the fiberwise linear Poisson structure of $(T^*M)/G$ induced by the Atiyah Lie algebroid  $(TM)/G\to M/G,$  i.e. ${\mathcal T}^{(TM)/G}{((T^*M)/G)}\to (T^*M)/G.$  In fact, this last Lie algebroid is just the Atiyah algebroid associated with the principal bundle $\pi_{T^*M}: T^*M\to (T^*M)/G$ (see \cite{LMM}) and its symplectic-like structure $\Omega_{(T^*M)/G}\in \Gamma(\wedge^2(T^*(T^*M)/G))$ is the one induced by the $G$-invariant symplectic structure on $T^*M.$    

\medskip

Now, we  choose  $\mu \in {\mathfrak g}^*$ such that $G=G_\mu$ and  a $G$-invariant $1$-form $\alpha_\mu\in \Omega^1(M)$ on $M$ such that $\alpha_\mu(M)\subset J^{-1}(\mu),$ where $J:T^*M\to  {\mathfrak g}^*$ is the momentum map given as in (\ref{cot momentum}). Then,  the reduced symplectic-like Lie algebroid $({\mathcal T}^AA^*)_\mu$  is simplectically isomorphic to  the Atiyah algebroid associated with the principal bundle $\pi_{T^*M}: T^*M\to (T^*M)/G$ endowed with the  
 symplectic-like structure 
 $$\Omega_{(T^*M)/G} - \gamma_\mu$$
 
where $\gamma_\mu\in \Gamma(\wedge^2(T^*M/G))$ is the $2$-section obtained from a magnetic term defined as follows.  
We consider  the epimorphism $$\widetilde{\pi}:{\mathfrak g}\times TM\to TM/G,\;\;\; \tilde\pi(\xi,v_x)=[v_x+\xi_M(x)].$$
Then, we have that there exists $B_\mu\in \Gamma(\wedge^2 ((T^*M)/G))$ such that $$\widetilde{\pi}^*(B_\mu)=(0,d\alpha_\mu).$$ Finally, $\gamma_\mu$ is just 
$$(T\tau_{T^*Q}/G)^*B_\mu,$$
where $T\tau_{T^*Q}/G:(T(T^*M))/G\to (TM)/G$ is the vector bundle induced by the equivariant tangent lift  
$T\tau_{T^*M}:T(T^*M)\to TM$ of $\tau_{T^*M}:T^*M\to M.$

}}\end{examples}

We finish  this paper with an application  which is related with the reduction of non-autonomous Hamiltonian systems. 

\begin{example}\label{p3}{\rm 
Let $p:M\to {\R}$ be a fibration. We denote by $\tau_{Vp}:Vp\to M$ the vertical bundle associated with $p$. Note that the sections of this vector bundle may be identified with the vector fields $X$ on $M$ such that $\eta(X)=0$, where $\eta$ is the exact $1$-form $p^*(dt)$ on $M$, $t$ being the standard coordinate on ${\R}.$ 

This vector bundle admits, in a natural way, a Lie algebroid structure where the Lie bracket is the standard Lie bracket of vector fields and the anchor map is the inclusion of vertical vectors with  respect to $p$ into $TM$. 

Now, suppose that  we additionally have a free and proper action $\phi:G\times M\to M$ of a Lie group $G$ on $M$ which is fibered, i.e.
$$p\circ \phi_g=p, \mbox{ for all } g\in G.$$ 

Then:
\begin{enumerate}
\item The infinitesimal generators of this last action  are vertical vector fields. 
\item The tangent lifted action $T\phi:G\times TM\to TM$ induces a free and proper action $$\Phi:G\times Vp\to Vp$$ of $G$  on the vertical vector bundle $Vp$ of $p.$ 
\item $p$ induces a new fibration $\widetilde{p}:M/G\to {\R}$ on the quotient manifold $M/G.$
\end{enumerate}
Then, $\Phi:G\times Vp\to Vp$ is an action by complete lifts with respect to the Lie algebra anti-morphism
$$\psi: {\mathfrak g}\to Vp, \;\;\; \psi(\xi)=\xi_{M}.$$

Let $\mu$ be an element of ${\mathfrak g}^*$ and we denote by  $J_{V^*p}:V^*p\to {\mathfrak g}^*$  the momentum map defined as in (\ref{A*}). Then we have that 
the vector bundle   ${\mathcal T}^{Vp}(J_{V^*p}^{-1}(\mu))\to J_{V^*p}^{-1}(\mu)$ may be identified in a natural way with the vertical bundle $Vp_\mu\to J_{V^*p}^{-1}(\mu)$, where $p_\mu: {J_{V^*p}^{-1}(\mu)}\to \R$ is the fibration given by
$$p_\mu=p\circ \tau_{J_{V^*p}^{-1}(\mu)},$$ 
$\tau_{J_{V^*p}^{-1}(\mu)}: {J_{V^*p}^{-1}(\mu)}\to M$ being the corresponding projection. 
Under this identification  the action $(\Phi, T\Phi^*):G_\mu\times  Vp_\mu\to Vp_\mu$ given by (\ref{ptp}) is described  as follows. Consider the tangent lift $T\Phi^*:G_\mu\times T(J_{V^*p}^{-1}(\mu))\to T(J_{V^*p}^{-1}(\mu))$ of the restricted dual action $\Phi^*:G_\mu\times J^{-1}_{V^*p}(\mu)\to J^{-1}_{V^*p}(\mu)$. Since 
$$p_\mu\circ \Phi^*=p_\mu,$$  we may induce an action of $G_\mu$ on $Vp_\mu$ which is just $(\Phi,T\Phi^*).$ 
Therefore,  the action $(\Phi, T\Phi^*)^T:TG_\mu\times Vp_\mu\to Vp_\mu$ is given by 

$$(\Phi,T\Phi^*)^T((g,\xi), v_{\alpha_x})=T_{\alpha_x}\Phi^*_g(v_{\alpha_x}+ \xi_M^{*c}(\alpha_x)),$$
for all $ (g,\xi)\in G_\mu\times {\mathfrak g}_\mu\cong TG_\mu\mbox{ and } \alpha_x\in V^*_xp.$ Here $\xi_M^{*c}\in {\mathfrak X}(V^*p)$ is the complete lift   to $V^*p$ of the  infinitesimal generator of $\xi$ with respect to the action $\phi.$ 

Finally, from Theorem \ref{2.8co}, we conclude that the reduced vector bundle 
$$(Vp_\mu)/TG_\mu\to J_{V^*p}^{-1}(\mu)/G_\mu$$
is a symplectic-like Lie algebroid. 

If $\mu=0$, then  using Theorem \ref{p1}, we have that this symplectic-like Lie algebroid is isomorphic to ${\mathcal T}^{A_0}A_0^*$ 
where $A_0$ is the quotient vector bundle over $M/G$ with total space $Vp/TG.$ We remark that the action of $TG\cong G\times {\mathfrak g}$ on $Vp$ is given by 
$$\Phi^T((g,\xi), v_x)=T_x\phi_g(v_x+\xi_{M}(x)),$$
for $(g,\xi)\in G\times {\mathfrak g}$ and $v_x\in V_xp.$
In fact, the vector bundle $A_0$ is isomorphic to the vertical bundle $V\widetilde{p}$ with respect to the fibration $\widetilde{p}:M/G\to {\R}$. The isomorphism is just 
$$Vp/TG\to V\widetilde{p},\;\;\; [v_x]\mapsto T_x\pi(v_x)$$
where $\pi:M\to M/G$ is the canonical projection. Therefore, ${\mathcal T}^{A_0}A_0^*$ may be identified with the vertical bundle $V\bar{p}\to V^*\widetilde{p}$ with 
$${\bar{p}}=\widetilde{p}\circ\tau_{V^*\widetilde{p}}: V^*\widetilde{p}\to \R,$$
 where $\tau_{V^*\widetilde{p}}: V^*\widetilde{p}\to M/G$ is the corresponding vector bundle projection. In conclusion,  the reduced Lie algebroid 
$(Vp_0)/TG\to J_{V^*p}^{-1}(0)/G$ is canonically isomorphic to the Lie algebroid $V\bar{p}$ on $V^*\widetilde{p}$  with its  standard symplectic-like structure. 

\medskip

Now, we consider $\mu\in {\mathfrak g}^*$ such that $G_\mu=G$. Let $\alpha_\mu$ be a $G$-invariant $1$-form on $M$ such that
$$\alpha_\mu(\xi_M)=\mu(\xi), \mbox{ for all } \xi\in {\mathfrak g}.$$

Then, the restriction $\alpha_{\mu|Vp}$ of $\alpha_\mu$ to the vertical bundle $Vp$ of the fibration $p:M\to \R$ determines a $G$-invariant section of $V^*p$ with values in $J_{V^*p}^{-1}(\mu).$

Let $\beta_\mu=d^{Vp}(\alpha_{\mu|Vp})$. Equivalently, $\beta_\mu$ is the restriction of $d\alpha_\mu\in \Omega^2(M)$ to $Vp\times Vp$. 
The magnetic term $B_\mu$  associated with $\alpha_\mu$ is the restriction to $V\widetilde{p}\times V\widetilde{p}$ of the unique  $2$-form $\bar{B}_\mu$ of $M/G$ such that 
$${\pi}^*\bar B_\mu=d\alpha_\mu$$
where $\pi:M\to M/G$ is the quotient projection.  Moreover, using Theorem  \ref{p2}, we have that $(Vp_\mu)/TG$  
is a symplectic-like Lie algebroid on $J_{V^*p}^{-1}(\mu)/G$ isomorphic to the Lie algebroid $V\bar{p}$ endowed with the symplectic-like section 
$$\Omega_{V\widetilde{p}}-pr_1^*B_\mu\in \Gamma(\wedge^2 V^*\bar{p})$$
with $pr_1:V\bar{p}\to V\widetilde{p}$  the vector bundle morphism on $\tau_{V^*\widetilde{p}}: V^*\widetilde{p}\to M/G$  given by the restriction to $V\bar{p}$  of the tangent lift $T\tau_{V^*\widetilde{p}}:T(V^*\widetilde{p})\to T(M/G)$.
}
\end{example}
\section{ Conclusions and future work}
In this paper we have proved a reduction theorem for Lie algebroids
with respect to a Lie group action by complete lifts. This result
allows to obtain a Lie algebroid version of the classic
Marsden-Weinstein reduction theorem for symplectic manifolds. We
remark that as for the usual Marsden-Weinstein reduction theorem,
the presence of the Lie group is superfluous, and the infinitesimal
action of its Lie algebra is sufficient, although we have chosen not
to use this approach here.

\medskip

Additionally, in this paper the Marsden-Weinstein reduction process
for symplectic-like Lie algebroids is applied to the particular case
of the canonical cover of a fiberwise Poisson structure. It would be
interesting to also obtain an analog version to the ``bundle" or
``fibrating" picture of cotangent bundle reduction in the setup of
symplectic-like Lie algebroids, but this will be studied elsewhere.

\medskip

It is also worth noticing that  the classical Marsden-Weinstein
reduction scheme does not only explain how to obtain a reduced
symplectic structure on a quotient manifold, but it also shows that
the reduced dynamics of a symmetric Hamiltonian function is again
Hamiltonian with respect to this reduced symplectic structure. It is
easy to prove that a similar phenomenon occurs for the reduction of
symplectic-like algebroids by complete actions. In fact, under the
same hypotheses as in Theorem  \ref{2.8co}, if $H:M\to\R$  is a
$G$-invariant Hamiltonian function,  then one can prove that the
restriction to $J^{-1}(\mu)$ of $H$ is a $G_\mu$-invariant function,
and thus, one can induce  a real smooth function ${H}_\mu$ on
$J^{-1}(\mu)/G_\mu.$ Moreover, the restriction of the Hamiltonian
section ${\mathcal H}_{H}^{\Omega}$ to $J^{-1}(\mu)$   is a section
of the Lie algebroid $(J^{T})^{-1}(0,\mu)\to J^{-1}(\mu)$ which is
$(\widetilde{\pi}_\mu,\pi_\mu)$-projectable on the Hamiltonian
section ${\mathcal H}_{{H}_\mu}^{\Omega_\mu}.$ Thus, if $\gamma:I\to
M$ is a solution of Hamilton's equations for $H$ on the
symplectic-like Lie algebroid $A\to M$ passing through a point in
$J^{-1}(\mu)$, then the curve $\gamma$ is contained in $J^{-1}(\mu)$
and $\pi_\mu\circ \gamma:I\to J^{-1}(\mu)/G_\mu$ is a solution of
Hamilton's equations for $H_\mu$ on the reduced symplectic-like Lie
algebroid $A_\mu \to J^{-1}(\mu)/G_\mu.$

In view of the results of this paper, one could apply this process
to the reduction of symmetric Hamiltonian systems on Poisson
manifolds. We have postponed this study for a future work.

 \end{document}